\documentclass[11pt,twoside,a4paper]{article}
\usepackage{color,amscd,amsmath,amssymb,amsthm,latexsym,stmaryrd,authblk,mathabx,shuffle,mathrsfs,hyperref,tikz,tkz-tab} 
\usepackage[latin1]{inputenc}
\usepackage[T1]{fontenc}   
\usepackage[english]{babel}
\providecommand{\keywords}[1]{\textbf{\textit{Keywords.}} #1}
\providecommand{\AMSclass}[1]{\textbf{\textit{AMS classification.}} #1}

\input{xy}
\xyoption{all}

\title{Cocommutative Com-PreLie bialgebras}
\date{}
\author{Lo\"ic Foissy}
\affil{\small{Univ. Littoral Côte d'Opale, UR 2597
LMPA, Laboratoire de Mathématiques Pures et Appliquées Joseph Liouville
F-62100 Calais, France}.\\ Email: \texttt{foissy@univ-littoral.fr}}

\theoremstyle{plain}
\newtheorem{theo}{Theorem}[section]
\newtheorem{lemma}[theo]{Lemma}
\newtheorem{cor}[theo]{Corollary}
\newtheorem{prop}[theo]{Proposition}
\newtheorem{defi}[theo]{Definition}

\theoremstyle{remark}
\newtheorem{remark}{Remark}[section]
\newtheorem{notation}{Notations}[section]

\newcommand{\K}{\mathbb{K}}

\newcommand{\A}{\K G\cdot S(V)}

\newcommand{\tdelta}{\tilde{\Delta}}

\newcommand{\g}{\mathfrak{g}}

\newcommand{\N}{\mathbb{N}}

\newcommand{\h}{\mathcal{H}}

\newcommand{\bftDelta}{\mathbf{\tdelta}}
\newcommand{\id}{\mathrm{Id}}
\newcommand{\Z}{\mathbb{Z}}
\newcommand{\prim}{\mathrm{Prim}}
\newcommand{\vect}{\mathrm{Vect}}
\newcommand{\der}{\mathrm{Der}}
\renewcommand{\ker}{\mathrm{Ker}}

\begin{document}

\maketitle

\begin{abstract}
A Com-PreLie bialgebra is a commutative bialgebra with an extra pre-Lie product satisfying some compatibilities with the product and the coproduct.
We here give a classification of connected, cocommutative Com-PreLie bialgebras over a field of characteristic zero: we obtain a main family 
of symmetric algebras on a space $V$ of any dimension, and another family available only if $V$ is one-dimensional.

We also explore the case of Com-PreLie bialgebras over a group algebra and over a tensor product of a group algebra and of a symmetric algebra.
\end{abstract}

\\

\AMSclass{17D25 16T05}

\tableofcontents

\section*{Introduction}

Pre-Lie algebras, also Vinberg, Gerstenhaber or left-symmetric algebras,were introduced by Vinberg \cite{Vinberg63} and Gerstenhaber \cite{Gerstenhaber63} in the sixties.
Typical examples are given by the flat and torsion free connection on a locally Euclidean space, see \cite{Floystad2018} for more details on these geometric aspects.
Chapoton and Livernet \cite{Chapoton2001} gave a construction of the operad of pre-Lie algebras, combinatorially described by rooted trees
with an operadic composition based on insertion of a tree at a vertex of another tree.
Examples of pre-Lie algebras occur in numerical analysis, Runge-Kutta methods and Butcher's series \cite{Brouder2004,Grossman89,Grossman2005}, 
Quantum Field Theory and Renormalization in the work of Connes and Kreimer \cite{Connes1998}, 
Ecalle's mould calculus and arborification's process \cite{Ebrahimi-Fard2017-2,Ecalle2004,Fauvet2017}, etc.
Com-PreLie bialgebras, introduced in \cite{Foissy28,Foissy30}, are commutative bialgebras with an extra pre-Lie product, compatible
with the product and coproduct, see Definition \ref{defi1} below. They appeared in Control Theory: the Lie algebra of the group of Fliess operators
\cite{Gray2011} naturally owns a Com-PreLie bialgebra structure, and its underlying bialgebra is a shuffle Hopf algebra.
Free (non unitary) Com-PreLie bialgebras were also described, in terms of partitioned rooted trees.\\

We here give examples of cocommutative Com-PreLie bialgebras, and, in particular, we classify all connected cocommutative (as coalgebras)
Com-PreLie bialgebras. We first introduce in Theorem \ref{theo2} a family $S(V,f,\lambda)$ of cocommutative and connected Com-PreLie bialgebras,
where $V$ is a vector space, $f$ a linear form on $V$ and $\lambda$ a scalar. These objects are classified up to an isomorphism
in Proposition \ref{prop4}. As a bialgebra, $S(V,f,\lambda)$ is the usual symmetric algebra on $V$ and, for any $x,x_1,\ldots,x_k\in V$,
\[x\bullet x_1\ldots x_k=\sum_{I\subsetneq [k]} |I|! \lambda^{|I|} f(x) \prod_{i\in I} f(x_i) \prod_{i\notin I} x_i.\]

Secondly, we give in Theorem \ref{theo5} all homogeneous prelie products on the polynomial algebra $\K[X]$, making it a Com-PreLie algebra:
we obtain four families. Among them, only a few satisfies the compatibility with the coproduct: we only obtain a one-parameter family
$\g^{(1)}(1,a,1)$, where $a$ is a scalar, see Proposition \ref{prop8}. For any $k,l\in \N$, in $\g^{(1)}(1,a,1)$,
\[X^k \bullet X^l=\frac{k}{l+1}X^{k+l}.\]
The underlying Lie algebras of these pre-Lie algebras are described in Proposition \ref{prop9} as semi-direct products of abelian
or Faà di Bruno Lie algebras.
We prove in Theorem \ref{theo9} that these examples cover all the connected cocommutative cases. Namely, if $A$ is a cocommutative Com-PreLie bialgebra, connected as a coalgebra, then it is isomorphic to $S(V,f,\lambda)$
or to $\g^{(1)}(1,a,1)$ (we should precise here that we work on a field of characteristic zero). 

We then turn to the non connected case and start with pre-Lie products on group algebras. We prove that if $G$ is an abelian group, then any pre-Lie
product $\bullet$ on $\K G$ making it a Com-PreLie bialgebra is given, for any $g,h\in G$, by
\[g\bullet h=\lambda(g,h)(g-gh),\]
where $(\lambda(g,h))_{g,h\in G}$ is a family of scalars satisfying certain conditions exposed in Theorem \ref{theo16}.
These conditions imply that if $G$ is a finite group, then $\bullet=0$. If $G=\Z$, we prove in Theorem \ref{theo18} that
 there exist two families of non trivial pre-Lie products on the Laurent polynomial algebra $\K\Z=\K[X,X^{-1}]$ making it a Com-PreLie bialgebra.
For the first one, there exist $k_0\in \Z$, nonzero, $a\in \K$, nonzero, such that
 \begin{align*}
& \forall k,l\in \Z,&X^k \bullet X^l&=a\delta_{l,k_0} k(X^k-X^{k+l}).
 \end{align*}
For the second one, there exist $\alpha,\beta \in \K\setminus \{0\}$, with a technical non vanishing condition, and there exists $N\geq 1$ such that
  \begin{align*}
& \forall k,l\in \Z,&X^k \bullet X^l&=\begin{cases}
\displaystyle \frac{\alpha\beta}{\left(\frac{l}{N}-1\right)\alpha-\left(\frac{l}{N}-2\right)\beta}k(X^k-X^{k+l})
\mbox{ if }N\mid l,\\
0\mbox{ otherwise}.
\end{cases}
 \end{align*}

We end by several results on the Hopf algebra $\K G\otimes S(V)$, where $G$ is an abelian group and $V$ a vector space.
In particular, we give in Theorem \ref{theo23} all possible pre-Lie products making it a Com-PreLie bialgebra,
with the extra conditions that $S(V)$ is a non trivial PreLie subalgebra, isomorphic to $S(V,f,\lambda)$. \\

This text is organized in six sections. The first one gives reminders and definitions on Com-PreLie bialgebras and Zinbiel-PreLie bialgebras.
The second one is devoted to the existence of Com-PreLie bialgebras $S(V,f,\lambda)$, and the third one to the classification of homogeneous
pre-Lie products on $\K[X]$. The theorem of classification of connected cocommutative Com-PreLie bialgebras is proved in the fourth section.
The study of pre-Lie products on a group algebra is done in the fifth section and the last one deals with the general case $\K G\otimes S(V)$.

\begin{notation} \begin{enumerate}
\item We denote by $\K$ a commutative field of characteristic zero. All the objects (vector spaces, algebras, coalgebras, prelie algebras$\ldots$)
in this text will be taken over $\K$.
\item For any $n\in \mathbb{N}$, we denote by $[n]$ the set $\{1,\ldots,n\}$.
\item Let $V$ be a vector space. We denote by $S(V)$ the symmetric algebra of $V$. It is a Hopf algebra, with the coproduct defined by
\begin{align*}
&\forall v\in V,&\Delta(v)&=v\otimes 1+1\otimes v.
\end{align*}
\end{enumerate}\end{notation}

 \section{Com-PreLie and Zinbiel-PreLie algebras}

\begin{notation} \end{notation}

\begin{defi}\label{defi1}
\begin{enumerate}
\item A \emph{Com-PreLie algebra} \cite{Mansuy2013}  is a family $A=(A,\cdot,\bullet)$, where $A$ is a vector space and 
$\cdot$, $\bullet$ are bilinear products on $A$, such that
\begin{align*}
&\forall a,b\in A,&a\cdot b&=b\cdot a,\\
&\forall a,b,c\in A,&(a\cdot b)\cdot c&=a\cdot (b\cdot c),\\
&\forall a,b,c\in A,&(a\bullet b)\bullet c-a\bullet(b\bullet c)&=(a\bullet c)\bullet b-a\bullet(c\bullet b)&\mbox{(pre-Lie identity)},\\
&\forall a,b,c\in A,&(a\cdot b)\bullet c&=(a\bullet c)\cdot b+a\cdot (b\bullet c)&\mbox{(Leibniz identity)}.
\end{align*}
In particular, $(A,\cdot)$ is an associative, commutative algebra and $(A,\bullet)$ is a right pre-Lie algebra.
We shall say that a Com-Prelie algebra is unitary if the associative algebra $(A,\cdot)$ has a unit, which will be denoted by $1$.
\item A \emph{Com-PreLie bialgebra} is a family $(A,\cdot,\bullet,\Delta)$, such that:
\begin{enumerate}
\item $(A,\cdot,\bullet)$ is a unitary Com-PreLie algebra.
\item $(A,\cdot,\Delta)$ is a bialgebra.
\item For any $a,b\in A$,
\[\Delta(a\bullet b)=a^{(1)}\otimes a^{(2)}\bullet b+a^{(1)}\bullet b^{(1)}\otimes a^{(2)}\cdot b^{(2)},\]
with Sweedler's notation $\Delta(x)=x^{(1)}\otimes x^{(2)}$.
\end{enumerate}
\item A \emph{Zinbiel-PreLie algebra}  is a family $A=(A,\prec,\bullet)$, where $A$ is a vector space and 
$\prec$, $\bullet$ are bilinear products on $A$, such that
\begin{align*}
&\forall a,b,c\in A,&(a\prec b) \prec c&=a\prec (b\prec c+c\prec b)&\mbox{(Zinbiel identity)},\\
&\forall a,b,c\in A,&(a\bullet b)\bullet c-a\bullet(b\bullet c)&=(a\bullet c)\bullet b-a\bullet(c\bullet b)&\mbox{(pre-Lie identity)},\\
&\forall a,b,c\in A,&(a\prec b)\bullet c&=(a\bullet c)\prec b+a\prec (b\bullet c)&\mbox{(Leibniz identity)}.
\end{align*}
In particular, $(A,\prec)$ is a Zinbiel algebra (or commutative half-shuffle algebra) \cite{Foissy21,Loday1995,Schutzenberger1958}.
The product $\cdot$ defined on $A$ by $a\cdot b=a\prec b+b\prec a$ is associative and commutative,
and $(A,\cdot,\prec)$ is a Com-PreLie algebra.
\item A \emph{Zinbiel-PreLie bialgebra} is a family $(A,\cdot,\prec,\bullet,\Delta)$ such that:
\begin{enumerate}
\item $(A,\cdot,\bullet,\Delta)$ is a Com-PreLie bialgebra. We denote by $A_+$ the augmentation ideal of $A$, and by $\tdelta$ the coassociative
coproduct defined by
\[\tdelta:\left\{\begin{array}{rcl}
A_+&\longrightarrow& A_+\otimes A_+\\
a&\longrightarrow& \Delta(a)-a\otimes 1-1\otimes a.
\end{array}\right.\]
\item $(A_+,\prec,\bullet)$ is a Zinbiel-PreLie algebra, and the restriction of $\cdot$ on $A_+$ is the commutative product induced by $\prec$:
for any $x,y \in A_+$, $x\prec y+y\prec x=x\cdot y$.
\item For any $a,b\in A_+$, with Sweedler's notation $\tdelta(x)=x'\otimes x''$, for any $a,b \in A_+$,
\[\tdelta(a\prec b)=a'\prec b'\otimes a''\cdot b''+a'\prec b\otimes a''+a'\otimes a''\cdot b+a\prec b'\otimes b''+a\otimes b.\]
\end{enumerate}\end{enumerate}\end{defi}

\begin{remark}\label{rk1.1} \begin{enumerate}
\item If $(A,\cdot,\bullet,\Delta)$ is a Com-PreLie bialgebra, then for any $\lambda \in \K$, $(A,\cdot,\lambda \bullet,\Delta)$ also is. 
\item If $(A,\prec,\bullet,\Delta)$ is Zinbiel-PreLie bialgebra, denoting $\cdot$ the product induced by $\prec$,
$(A,\cdot,\bullet,\Delta)$ is a Com-PreLie bialgebra.
\item If $A$ is a Zinbiel-PreLie bialgebra, we extend $\prec$ to $A_+\otimes A+A\otimes A_+$ 
by $a\prec 1=a$ and $1\prec a=0$ for any $a\in A_+$. Note that $1\prec 1$ is not defined. 
\item If $(A,\cdot,\bullet)$ is a unitary Com-PreLie algebra, for any $x\in A$,
\begin{align*}
1\bullet x&=(1\cdot 1)\bullet x=(1\bullet x)\cdot 1+1\cdot (1\bullet x)=2(1\bullet x).
\end{align*}
Hence, for any $x\in A$, $1\bullet x=0$.
\item If $(A,\cdot,\bullet,\Delta)$ is a Com-PreLie bialgebra, we denote by $\prim(A)$ the subspace of primitive elements of $A$, that is to say
\[\prim(A)=\{x\in A\mid \Delta(x)=x\otimes 1+1\otimes x\}.\]
For any $x\in \prim(A)$,
\begin{align*}
\Delta(x\bullet 1)&=x\otimes 1\bullet 1+1\otimes x\bullet 1+1\otimes x\bullet 1+x\otimes 1\bullet 1=1\otimes x\bullet 1+1\otimes x\bullet 1.
\end{align*}
So $x\bullet 1\in \prim(A)$. We shall consider the map
\begin{align*}
f_A&:\left\{\begin{array}{rcl}
\prim(A)&\longrightarrow&\prim(A)\\
x&\longrightarrow&x\bullet 1.
\end{array}\right.
\end{align*}
\end{enumerate}\end{remark}

\section{Examples on symmetric algebras}

Our goal in this section is to prove the following theorem:

\begin{theo}\label{theo2}
Let $V$ be a vector space, $f\in V^*$, $\lambda \in \K$. We give $S(V)$ the product $\bullet$ defined by
\begin{align*}
&\forall x\in S(V),&1\bullet x&=0,\\
&\forall x,x_1,\ldots,x_k \in V,&x\bullet x_1\ldots x_k&=\sum_{I\subsetneq [k]} |I|! \lambda^{|I|} f(x) \prod_{i\in I} f(x_i) \prod_{i\notin I} x_i,\\
&\forall x_1,\ldots,x_k \in V,\:\forall x\in S(V),&x_1\ldots x_k\bullet x&=\sum_{i=1}^k x_1\ldots (x_i \bullet x)\ldots x_k.
\end{align*}
Then $(S(V),m,\bullet,\Delta)$ is a Com-PreLie bialgebra, denoted by $S(V,f,\lambda)$. 
\end{theo}

\subsection{Two operators}

Let us fix a vector space $V$ and $f\in V^*$. We shall consider the two following operators:
\begin{align*}
\partial:&\left\{\begin{array}{rcl}
S(V)&\longrightarrow& S(V)\\
x_1\ldots x_k&\longrightarrow&\displaystyle \sum_{i=1}^k x_1\ldots x_{i-1}f(x_i)x_{i+1}\ldots x_k,
\end{array}\right.\\
\phi:&\left\{\begin{array}{rcl}
S(V)&\longrightarrow& S(V)\\
x_1\ldots x_k&\longrightarrow&\displaystyle\sum_{I\subsetneq [k]}\lambda^{|I|} |I|! \prod_{i\in I} f(x_i) \prod_{i\notin I} x_i,
\end{array}\right. \end{align*} 
where $x_1,\ldots,x_k$ are elements of $V$.

\begin{lemma} \label{lem3}\begin{enumerate}
\item For any $u,v,w \in S(V)$,
\begin{align*}
\partial(uv)&=\partial(u)v+u\partial(v),\\
\partial \circ \phi(v)\phi(w)-\phi(\partial(v)\phi(w))&=\partial \circ \phi(w)\phi(v)-\phi(\partial(w)\phi(v)).
\end{align*}
\item For any $u\in S(V)$,
\begin{align*}
\Delta \circ \partial(u)&=(\partial \otimes \id_{S(V)})\circ \Delta(u)=(\id_{S(V)} \otimes \partial)\circ \Delta(u),\\
\Delta \circ \phi(u)&=(\phi \otimes \id_{S(V)})\circ \Delta(u)+1\otimes \phi(u).
\end{align*}\end{enumerate}\end{lemma}

\begin{proof}
1. The fact that $\delta$ is a derivation is immediate. Let us prove the second assertion. We consider $v=x_1\ldots x_k$ and $w=y_1\ldots y_l$,
with $x_1,\ldots,x_k$ and $y_1,\ldots,y_l\in V$. Then
\begin{align*}
\partial \circ \phi(v)\phi(w)&=\sum_{\substack{I\subsetneq [k],\: i_0\notin I,\\ J\subsetneq [l]}}
\lambda^{|I|+|J|}|I|!|J|! \prod_{i\in I\sqcup\{i_0\}} f(x_i) \prod_{j\in J} f(y_j)\prod_{i\notin I\sqcup \{i_0\}} x_i\prod_{j\notin J}y_j\\
&=\sum_{\substack{\emptyset \subsetneq I\subseteq [k],\\ J\subsetneq [l]}}
\lambda^{|I|+|J|-1}|I|!|J|! \prod_{i\in I} f(x_i) \prod_{j\in J} f(y_j)\prod_{i\notin I} x_i\prod_{j\notin J}y_j\\
&=\underbrace{\sum_{\substack{\emptyset \subsetneq I\subsetneq [k],\\ \emptyset \subsetneq J\subsetneq [l]}}
\lambda^{|I|+|J|-1}|I|!|J|! \prod_{i\in I} f(x_i) \prod_{j\in J} f(y_j)\prod_{i\notin I} x_i\prod_{j\notin J}y_j}_{=\varphi_1(v,w)}\\
&+\sum_{\emptyset \subsetneq I\subsetneq [k]}
\lambda^{|I|-1}|I|! \prod_{i\in I} f(x_i) \prod_{i\notin I} x_i\prod_{j\in [l]}y_j\\
&+\sum_{\emptyset \subsetneq J\subsetneq [l]}
\lambda^{k+|J|-1}k!|J|! \prod_{i\in [k]} f(x_i) \prod_{j\in J} f(y_j)\prod_{j\notin J}y_j\\
&+\lambda^{k-1} k! \prod_{i\in [k]} f(x_i) \prod_{j\in [l]} y_j.
\end{align*}
Observe that $\varphi_1(v,w)=\varphi_1(w,v)$. 
\begin{align*}
&\phi(\partial(v)\phi(w))\\
&=\sum_{\substack{i_0\in [k],\: I\subseteq [k]\setminus\{i_0\}\\
J'\subsetneq [l],\: J''\subsetneq [l]\setminus J,\\ I\neq [k]\setminus \{i_0\}\mbox{\scriptsize{ or }}J''\neq [l]\setminus J}}
\lambda^{|I|+|J'|+|J''|} |J'|(|I|+|J''|)!\prod_{i\in I\sqcup \{i_0\}}f(x_i) \prod_{j\in J'\sqcup J''} f(y_j)
\prod_{i\notin I\sqcup \{i_0\}}x_i \prod_{j\notin J'\sqcup J''} y_j\\
&=\sum_{\substack{\emptyset\subsetneq I\subseteq [k],\\ J\subseteq [l]}}\lambda^{|I|+|J|-1}|I|\left(\sum_{\substack{J=J'\sqcup J'',\\ J'\neq [l]}}
|J'|!(|J''|+|I|-1)!\right)\prod_{i\in I} f(x_i) \prod_{j\in J} f(y_j) \prod_{i\notin I} x_i \prod_{j\notin J} y_j.
\end{align*}
For any $I\subseteq [k]$ and $J\subseteq [l]$,
\begin{align*}
|I|\left(\sum_{J=J'\sqcup J''}|J'|!(|J''|+|I|-1)!\right)&=
|I|\sum_{k=0}^{|J|}\binom{|J|}{k}k!(|I|+|J|-1)!\\
&=|I|!|J|!\sum_{k=0}^{|J|}\binom{k}{|I|-1}\\
&=|I|!|J|!\sum_{k=|I|-1}^{|I|+|J|-1}\binom{k}{|I|-1}\\
&=|I|!|J|!\binom{|I|+|J|}{|I|}\\
&=(|I|+|J|)!,
\end{align*}
and
\begin{align*}
|I|\left(\sum_{\substack{J=J'\sqcup J'',\\ J'\neq [l]}}|J'|!(|J''|+|I|-1)!\right)&=\begin{cases}
(|I|+|J|)! \mbox{ if }J\neq [l],\\
(|I|+|J|)!-|I|!|J|!\mbox{ if }J=[l].
\end{cases}
\end{align*}
This gives
\begin{align*}
\phi(\partial(v)\phi(w))&=\underbrace{\sum_{\substack{\emptyset\subsetneq I\subset [k],\\\emptyset\subsetneq J\subset [l],\\
I\neq [k]\mbox{\scriptsize{ or }}J\neq [l]}}
\lambda^{|I|+|J|-1}(|I|+|J|)!\prod_{i\in I}f(x_i)\prod_{j\in J}f(y_j) \prod_{i\notin I}x_i\prod_{j\notin J}y_j}_{=\varphi_2(v,w)}\\
&+\sum_{\emptyset \subsetneq I\subsetneq [k]}\lambda^{|I|-1}|I|! \prod_{i\in I} f(x_i) \prod_{i\notin I} x_i\prod_{j\in [l]}y_j\\
&+\lambda^{k-1} k! \prod_{i\in [k]} f(x_i) \prod_{j\in [l]} y_j\\
&-\sum_{\emptyset \subsetneq I\subsetneq [k]}\lambda^{l+|I|-1}l!|I|!\prod_{i\in I}f(x_i)\prod_{j\in [l]}f(y_j)\prod_{i\notin I}y_j.
\end{align*}
Note that $\varphi_2(v,w)=\varphi_2(w,v)$. Finally,
\begin{align*}
\partial \circ \phi(v)\phi(w)-\phi(\partial(v)\phi(w))&=\varphi_1(v,w)-\varphi_2(v,w)\\
&+\sum_{\emptyset \subsetneq J\subsetneq [l]}\lambda^{k+|J|-1}k!|J|! \prod_{i\in [k]} f(x_i) \prod_{j\in J} f(y_j)\prod_{j\notin J}y_j\\
&+\sum_{\emptyset \subsetneq I\subsetneq [k]}\lambda^{l+|I|-1}l!|I|!\prod_{i\in I}f(x_i)\prod_{j\in [l]}f(y_j)\prod_{i\notin I}y_j.
\end{align*}
This is symmetric in $v,w$.\\

2. Let us consider $A=\{u\in S(V)\mid \Delta\circ\partial u=(\partial \otimes \id_{S(V)})\circ \Delta(u)\}$.
As $\partial(1)=0$, $1\in A$. If $x\in V$,
\[\Delta\circ \partial(x)=f(x)1\otimes 1=\partial(x)\otimes 1+\partial(1)\otimes x=(\partial \otimes \id_{S(V)})\circ \Delta(x),\]
so $V\subseteq A$. Let $u,v\in A$.
\begin{align*}
\Delta\circ \partial(uv)&=\Delta(\partial(u)v+u\partial(v))\\
&=\partial(u^{(1)})v^{(1)}\otimes u^{(2)}v^{(2)}+u^{(1)}\partial(v^{(1)})\otimes u^{(2)}v^{(2)}\\
&=\partial(u^{(1)}v^{(1)})\otimes u^{(2)}v^{(2)}\\
&=(\partial \otimes \id_{S(V)})\circ \Delta(uv).
\end{align*}
We proved that $A$ is a subalgebra of $S(V)$ containing $V$, so $A=S(V)$.\\

We put
\[\tau:\left\{\begin{array}{rcl}
S(V)\otimes S(V)&\longrightarrow&S(V)\otimes S(V)\\
a\otimes b&\longmapsto&b\otimes a.
\end{array}\right.\]
As $\Delta$ is cocommutative,
\begin{align*}
\Delta\circ \partial&=\tau\circ \Delta \circ \partial\\
&=\tau \circ (\partial \otimes \id_{S(V)})\circ \Delta\\
&=(\id_{S(V)} \otimes \partial)\circ\tau \circ\Delta\\
&=(\id_{S(V)} \otimes \partial) \circ \Delta.
\end{align*}

Let $u=x_1\ldots x_k\in S(V)$. 
\begin{align*}
\Delta \circ \phi(u)&=\sum_{\substack{[k]=I\sqcup J\sqcup K,\\J\sqcup K\neq \emptyset}}\lambda^{|I|}|I|! \prod_{i\in I}f(x_i)
\prod_{j\in J} x_j \otimes \prod_{k\in K} x_k\\
&=\sum_{\substack{[k]=I\sqcup J\sqcup K,\\J\neq \emptyset}}\lambda^{|I|}|I|! \prod_{i\in I}f(x_i) \prod_{j\in J} x_j \otimes \prod_{k\in K} x_k
+\sum_{\substack{[k]=I\sqcup K,\\ K\neq \emptyset}}\lambda^{|I|}|I|! \prod_{i\in I}f(x_i) 1 \otimes \prod_{k\in K} x_k\\
&=\sum_{[k]=I\sqcup K} \phi\left(\prod_{i\in I} x_i\right)\otimes \prod_{k\in K} x_k+1\otimes \phi(u)\\
&=(\phi\otimes \id_{S(V)})\circ \Delta(u)+1\otimes \phi(u),\end{align*}
which ends this proof. \end{proof}

\subsection{Proof of Theorem \ref{theo2}}

To start with, observe that for any $u,v\in S(V)$,
\[u\bullet v=\partial(u)\phi(v).\]

1. We first prove the Leibniz identity. Let us take $u,v,w\in S(V)$. As $\partial$ is a derivation,
\begin{align*}
(uv)\bullet w&=\partial(uv)\phi(w)=\partial(u)v\phi(w)+u\partial(v)\phi(w)=(u\bullet w)v+u(v\bullet w).
\end{align*} 

2. Let us now prove the pre-Lie identity. If $u,v,w\in S(V)$,
\begin{align*}
(u\bullet v)\bullet w-u\bullet (v\bullet w)&=\partial(\partial(u)\phi(v))\phi(w)-\partial(u)\phi(\partial(v)\phi(w))\\
&=\partial^2(u)\phi(v)\phi(w)+\partial(u)(\partial\circ \phi(v)\phi(w)-\phi(\partial(v)\phi(w))).
\end{align*}
By Lemma \ref{lem3}, this is symmetric in $v,w$.\\

3. Let us finish by the compatibility with the coproduct. For any $u,v\in S(V)$, by Lemma \ref{lem3},
\begin{align*}
\Delta(u\bullet v)&=\Delta(\partial(u)\phi(v))\\
&=\Delta\circ \partial(u)(\phi(v^{(1)})\otimes v^{(2)}+1\otimes \phi(v))\\
&=\partial(u^{(1)})\phi(v^{(1)})\otimes u^{(2)}v^{(2)}+u^{(1)}\otimes \partial(u^{(2)})\phi(v)\\
&=u^{(1)}\bullet v^{(1)}\otimes u^{(2)}v^{(2)}+u^{(1)}\otimes u^{(2)}\bullet v.
\end{align*}
Hence, $S(V,f,\lambda)$ is indeed a Com-PreLie bialgebra.

\subsection{Isomorphisms}

\begin{prop}\label{prop4}
Let $V,W$ be two vector spaces, $f$ and $g$ be linear forms on respectively $V$ and $W$, and $\lambda,\mu\in \K$.
The Com-PreLie bialgebras $S(V,f,\lambda)$ and $S(W,g,\mu)$ are isomorphic if, and only if, one of the two following assertions holds:
\begin{enumerate}
\item $f=g=0$ and $\dim(V)=\dim(W)$.
\item  $\lambda=\mu$ and there exists a linear bijection $\psi:V\longrightarrow W$ such that $g\circ \psi=f$.
\end{enumerate}
\end{prop}

\begin{proof} If the first assertion holds, then both pre-Lie products on $S(V,f,\lambda)$ and $S(W,g,\mu)$ are zero. Any linear isomorphism
between $V$ and $W$, extended as an algebra isomorphism from $S(V)$ to $S(W)$, is a Com-PreLie bialgebra isomorphism.\\

If the second assertion holds, the extension of $\psi$ as an algebra isomorphism from $S(V)$ to $S(W)$ is a Com-PreLie bialgebra isomorphism.\\

Let $\Psi:S(V,f,\lambda)\longrightarrow S(W,g,\mu)$ be an isomorphism. It is a bialgebra isomorphism,
so the restriction $\psi$ of $\Psi$ to $\prim(S(V))=V$ is a bijection to $\prim(S(W))=W$. As $\Psi$ is an algebra morphism,
it is the extension of $\psi$ as an algebra morphism from $S(V)$ to $S(W)$.

Let $x,y\in V$. Then
\begin{align*}
\Psi(x\bullet y)&=\Psi(f(x)y)=f(x)\psi(y)\\
=\Psi(x)\bullet \Psi(y)&=g\circ \psi(x) \psi(y).
\end{align*}
Choosing a nonzero $y$, this proves that $f=g\circ \psi$. As a consequence, $f=0$ if, and only if, $g=0$.

Let $x,y,z\in V$. Then
\begin{align*}
\Psi(x\bullet yz)&=f(x)\psi(y)\psi(z)+\lambda f(x)f(y)\psi(z)+\lambda f(x)f(z)\psi(y),\\
=\psi(x)\bullet \psi(y)\psi(z)&=f(x)\psi(y)\psi(z)+\mu f(x)f(y)\psi(z)+\mu f(x)f(z)\psi(y)
\end{align*}
If $f\neq 0$, let us choose $x=y=z$ such that $f(x)=1$. Then $2\lambda \psi(x)=2\mu \psi(x)$, so $\lambda=\mu$. \end{proof}

\begin{remark}
If $\bullet$ is the product of $S(V,f,\lambda)$ and $\mu \neq 0$,
the Com-PreLie bialgebra $(S(V),m,\mu \bullet,\Delta)$ is $S(V,\mu f,\lambda/\mu)$.
\end{remark}

\section{Examples on $\K[X]$}

Our aim in this section is to give all pre-Lie products on $\K[X]$, making it a graded Com-PreLie algebra.\\

\begin{notation}
Recall that $\K[X]$ is given a Zinbiel product $\prec$, defined by
\begin{align*}
&\forall i,j\geq 1,&X^i\prec X^j&=\frac{i}{i+j}X^{i+j}.
\end{align*}
The associated product is the usual product of $\K[X]$.
\end{notation}

We shall prove the following result:

\begin{theo}\label{theo5}
The following objects are Zinbiel-PreLie algebras:
\begin{enumerate}
\item Let $N\geq 1$, $\lambda,a,b \in \K$, $a\neq 0$, $b\notin \mathbb{Z}_-$. We put $\g^{(1)}(N,\lambda,a,b)=(\K[X],m,\bullet)$, with
\[X^i \bullet X^j=\begin{cases}
i\lambda X^i \mbox{ if }j=0,\\
a\dfrac{i}{\frac{j}{N}+b} X^{i+j}\mbox{ if }j\neq 0 \mbox{ and }N\mid j,\\
0\mbox{ otherwise}.
\end{cases}\]
\item Let $N\geq 1$, $\lambda,\mu \in \K$, $\mu \neq 0$. We put $\g^{(2)}(N,\lambda,\mu)=(\K[X],m,\bullet)$, with
\[X^i \bullet X^j=\begin{cases}
i\lambda X^i \mbox{ if }j=0,\\
i\mu X^{i+N} \mbox{ if }j=N,\\
0\mbox{ otherwise}.
\end{cases}\]
\item Let $N\geq 1$, $\lambda,\mu \in \K$, $\mu \neq 0$. We put $\g^{(3)}(N,\lambda,\mu)=(\K[X],m,\bullet)$, with
\[X^i \bullet X^j=\begin{cases}
i\lambda X^i \mbox{ if }j=0,\\
i\mu X^{i+j}  \mbox{ if }j\neq 0 \mbox{ and }N\mid j,\\
0\mbox{ otherwise}.
\end{cases}\]
\item Let $\lambda \in \K$. We put $\g^{(4)}(\lambda)=(\K[X],m,\bullet)$, with
\[X^i \bullet X^j=\begin{cases}
i\lambda X^i \mbox{ if }j=0,\\
0\mbox{ otherwise}.
\end{cases}\]
In particular, the pre-Lie product of $\g^{(4)}(0)$ is zero.
\end{enumerate}
Moreover, if  $\bullet$ is a product on $\K[X]$, such that $\g=(\K[X],m,\bullet)$ is a graded Com-PreLie algebra, Then $\g$ is one of the 
preceding examples. \end{theo}

\begin{remark} If $\lambda=\dfrac{a}{b}$, in $\g^{(1)}(N,\lambda,a,b)$,
\[X^i \bullet X^j=\begin{cases}
\dfrac{ai}{\frac{j}{N}+b} X^{i+j}\mbox{ if }N\mid j,\\
0\mbox{ otherwise}.
\end{cases}\]
We denote $\g^{(1)}(N,a,b)=\g^{(1)}\left(N,\dfrac{a}{b},a,b\right)$.
\end{remark}

\subsection{Graded pre-Lie products on $\K[X]$}

In this paragraph, we look for all graded pre-Lie products on $\K[X]$, making it a Com-PreLie algebra.
Let $\bullet$  be a homogeneous product on $\K[X]$, making it a graded Com-PreLie algebra.
For any $i,j\in \N$, there exists a scalar $\lambda_{i,j}$ such that
\[X^i\bullet X^j=\lambda_{i,j}X^{i+j}.\]
Moreover, for any $i,j,k \in \N$,
\begin{align*}
X^{i+j}\bullet X^k&=\lambda_{i+j,k}X^{i+j+k}\\
&=(X^iX^j)\bullet X^k\\
&=(X^i \bullet X^k)X^j+X^i(X^j \bullet X^k)\\
&=(\lambda_{i,k}+\lambda_{j,k})X^{i+j+k}.
\end{align*}
Hence, $\lambda_{i+j,k}=\lambda_{i,k}+\lambda_{j,k}$. Putting $\lambda_k=\lambda_{1,k}$ for any $k\in \N$, we obtain that
\[X^i\bullet X^j=i\lambda_j X^{i+j}.\]

\begin{lemma} \label{lem6}
For any $k\in \N$, let $\lambda_k \in \K$. We define a product $\bullet$ on $\K[X]$ by
\[X^i\bullet X^j=i\lambda_j X^{i+j}.\]
Then $(\K[X],m,\bullet)$ is Com-PreLie if, and only if, for any $j,k\geq 1$,
\[(j\lambda_k-k\lambda_j)\lambda_{j+k}=(j-k)\lambda_j\lambda_k.\]
\end{lemma}

\begin{proof} Let $i,j,k\in \N$. Then
\[X^i\bullet(X^j \bullet X^k)-(X^i\bullet X^j)\bullet X^k=(ij \lambda_k \lambda_{j+k}-i(i+j)\lambda_j\lambda_k)X^{i+j+k}.\]
So $\bullet$ is pre-Lie if, and only if
\begin{align*}
&\forall i,j,k\in \N, ij \lambda_k \lambda_{j+k}-i(i+j)\lambda_j\lambda_k=ik \lambda_j \lambda_{j+k}-i(i+k)\lambda_j\lambda_k\\
&\Longleftrightarrow \forall j,k\in \N, (j\lambda_k-k\lambda_j)\lambda_{j+k}=(j-k)\lambda_j\lambda_k\\
&\Longleftrightarrow \forall j,k\geq 1, (j\lambda_k-k\lambda_j)\lambda_{j+k}=(j-k)\lambda_j\lambda_k,
\end{align*} 
as the identity is trivially satisfied if $j=0$ or $k=0$. \end{proof}

\begin{lemma}\label{lem7}
Let $\bullet$ be a product on $\K[X]$, making it a graded Com-PreLie algebra. Then $(\K[X],\prec,\bullet)$ is a Zinbiel-PreLie algebra.
\end{lemma}

\begin{proof} Let us take $i,j,k\in \N$, $(i,j)\neq(0,0)$. Then
\begin{align*}
(X^i\bullet X^k)\prec X^j+X^i \prec (X^j \bullet X^k)&=\lambda_k(i X^{i+k}\prec X^j+j X^i \prec X^{j+k})\\
&=\lambda_k\left(\frac{i(i+k)}{i+j+k}+\frac{ij}{i+j+k}\right)X^{i+j+k}\\
&=i\lambda_kX^{i+j+k}\\
&=(i+j)\lambda_k \frac{i}{i+j} X^{i+j+k},\\
(X^i\prec X^j)\bullet X^k&=\frac{i}{i+j}X^{i+j}\bullet X^k\\
&=\frac{i}{i+j}(i+j)\lambda_k X^{i+j+k}\\
&=i\lambda_k X^{i+j+k}.
\end{align*}
So $\K[X]$ is Zinbiel-PreLie. \end{proof}

\begin{proof} (Theorem \ref{theo5}-1). Let us first prove that the objects defined in Theorem \ref{theo5} are indeed Zinbiel-PreLie algebras.
By Lemma \ref{lem7}, it is enough to prove that they are Com-PreLie algebras. We shall use Lemma \ref{lem6} in all cases. \\

1. For any $j\geq 1$, $\lambda_j=a\frac{1}{\frac{j}{N}+b}$ if $N\mid j$ and $0$ otherwise. If $j$ or $k$ is not a multiple of $N$, then
\[(j\lambda_k-k\lambda_j)\lambda_{j+k}=(j-k)\lambda_j\lambda_k=0.\]
If $j=Nj'$ and $k=Nk'$, with $j',k'\in \N$, then
\begin{align*}
(j\lambda_k-k\lambda_j)\lambda_{j+k}&=Na^2\left(\frac{j'}{k'+b}-\frac{k'}{j'+b}\right)\frac{1}{j'+k'+b}\\
&=Na^2\frac{j'^2-k'^2+b(j'-k')}{(j'+b)(k'+b)(j'+k'+b)}\\
&=Na^2(j'-k')\frac{j'+k'+b}{(j'+b)(k'+b)(j'+k'+b)}\\
&a^2(j-k)\frac{1}{(j'+b)(k'+b)}\\
&=(j-k)\lambda_j\lambda_k.
\end{align*}

2. In this case, $\lambda_j=\mu$ if $j=N$ and $0$ otherwise. Hence, for any $j,k \geq 1$,
\begin{align*}
(j\lambda_k-k\lambda_j)\lambda_{j+k}&=\mu^2(j\delta_{k,N}-k\delta_{j,N})\delta_{j+k,N}=0,\\
(j-k)\lambda_j\lambda_k&=\mu^2(j-k) \delta_{j,N} \delta_{k,N}=0.
\end{align*}

3. Here, for any $j\geq 1$, $\lambda_j=\mu$ if $N\mid j$ and $0$ otherwise. Then
\begin{align*}
(j\lambda_k-k\lambda_j)\lambda_{j+k}&=\begin{cases}
\mu^2(j-k) \mbox{ if }N\mid j,k,\\
0\mbox{ otherwise};
\end{cases}&
(j-k)\lambda_j\lambda_k&=\begin{cases}
\mu^2(j-k) \mbox{ if }N\mid j,k,\\
0\mbox{ otherwise}.
\end{cases}\end{align*}

4. In this case, for any $j\geq 1$, $\lambda_j=0$ and the result is trivial. \end{proof}

\subsection{Classification of graded pre-Lie products on $\K[X]$}

We now prove that the preceding examples cover all the possible cases.\\

\begin{proof} (Theorem \ref{theo5}-2).  We put $X^i \bullet X^j=i\lambda_j X^{i+j}$ for any $i,j \in \N$ and we put $\lambda=\lambda_0$. 
If for any $j\geq 1$, $\lambda_j=0$, then $\g=\g^{(4)}(\lambda)$. If this is not the case, we put
\[N=\min\{j\geq 1\mid \lambda_j \neq 0\}.\]

\textit{First step.} Let us prove that if $i$ is not a multiple of $N$, then $\lambda_i=0$. If $i$ is not a multiple of $N$,
we put $i=qN+r$, with $0<r<N$, and we proceed by induction on $q$. If $q=0$, by definition of $N$,
$\lambda_1=\ldots=\lambda_{N-1}=0$. Let us assume the result at rank $q-1$, with $q>0$. 
We put $j=i-N$ and $k=N$. By the induction hypothesis, $\lambda_j=0$. Then, by Lemma \ref{lem6},
\[(i-N) \lambda_N\lambda_i=0.\]
As $i\neq N$ and $\lambda_N \neq 0$, $\lambda_i=0$. It is now enough to determine $\lambda_{iN}$ for any $i\geq 1$. \\

\textit{Second step.} Let us assume that $\lambda_{2N}=0$. Let us prove that $\lambda_{iN}=0$ for any $i\geq 2$, by induction on $i$.
This is obvious if $i=2$. Let us assume the result at rank $i-1$, with $i\geq 3$, and let us prove it at rank $i$. 
We put $j=(i-1)N$ and $k=N$. By the induction hypothesis, $\lambda_j=0$. Then, by Lemma \ref{lem6},
\[(i-2)N\lambda_N \lambda_iN=0.\]
As $i\geq 3$ and $\lambda_N \neq 0$, $\lambda_{iN}=0$. As a conclusion, if $\lambda_{2N}=0$, putting $\mu=\lambda_N$, $\g=\g^{(2)}(N,\lambda,\mu)$. \\

\textit{Third step.} We now assume that $\lambda_{2N} \neq 0$. We first prove that $\lambda_{iN}\neq 0$ for any $i\geq 1$.
This is obvious if $i=1,2$. Let us assume the result at rank $i-1$, with $i\geq 3$, and let us prove it at rank $i$. 
We put $j=(i-1)N$ and $k=N$.  Then, by Lemma \ref{lem6},
\[(j\lambda_N-N\lambda_j)\lambda_{iN}=(i-2)N\lambda_j\lambda_N.\]
By the induction hypothesis, $\lambda_j\neq 0$. Moreover, $i>2$ and $\lambda_N\neq 0$, so $\lambda_{iN} \neq 0$.\\

For any $j\geq 1$, we put $\mu_j=\dfrac{\lambda_{kN}}{\lambda_N}$: this is a nonzero scalar, and $\mu_1=1$. 
Let us prove inductively that
\begin{align*}
\mu_k&=\frac{\mu_2}{(k-1)-(k-2)\mu_2},&\mu_2&\neq \frac{k-1}{k-2} \mbox{ if }k\neq 2.
\end{align*}
If $k=1$, $\mu_1=1=\dfrac{\mu_2}{0-(-1)\mu_2}$, and $\mu_2\neq 0$ as $\lambda_{2N}\neq 0$.
If $k=2$, $\mu_2=\dfrac{\mu_2}{1-0\mu_2}$. Let us assume the result at rank $k-1$, with $k\geq 3$. 
By Lemma \ref{lem6}, with $j=(k-1)N$ and $k=N$,
\begin{align*}
((k-1)N \lambda_N-\lambda_N\mu_{k-1})\lambda_N\mu_k&=(k-2)N \mu_{k-1}\mu_1\lambda_N^2,\\
\mu_k((k-1)-\mu_{k-1})&=(k-2)\mu_{k-1}.
\end{align*}
Moreover, by the induction hypothesis,
\begin{align*}
(k-1)-\mu_{k-1}&=k-1-\frac{\mu_2}{(k-2)-(k-3)\mu_2}\\
&=\frac{(k-1)(k-2)-((k-1)(k-3)+1)\mu_2}{(k-2)-(k-3)\mu_2}\\
&=(k-2)\frac{(k-1)-(k-2)\mu_2}{(k-2)-(k-3)\mu_2}.
\end{align*}
As $\mu_{k-1}\neq 0$ and $k>2$, this is nonzero, so $\mu_2\neq \dfrac{k-1}{k-2}$. We finally obtain that
\[\mu_k=(k-2)\mu_{k-1}\frac{1}{k-2}\frac{(k-2)-(k-3)\mu_2}{(k-1)-(k-2)\mu_2}=\frac{\mu_2}{(k-1)-(k-2)\mu_2}.\]
Finally, for any $k\geq 1$,
\[\lambda_{kN}=\frac{a\mu_2}{(k-1)-(k-2)\mu_2}=\frac{\lambda_N\mu_2}{(1-\mu_2)k+2\mu_2-1}.\]

\textit{Last step.} If $\mu_2=1$, then for any $k\geq 1$, $\lambda_{kN}=\lambda_N$: this is $\g^{(3)}(N,\lambda,\lambda_N)$. If $\mu_2\neq 1$,
we put  $b=\dfrac{2\mu_2-1}{1-\mu_2}$. As $\mu_2\neq 0$, $b\neq -1$. As for any $k\geq 3$, $\mu_2\neq \dfrac{k-1}{k-2}$,
$b\neq -k$, and $b\neq -2$, so $b\notin \mathbb{Z}_-$. Moreover, for any $k\geq 1$,
\[\lambda_{kN}=\dfrac{\dfrac{\lambda_N\mu_2}{1-\mu_2}}{k+b}.\]
We take $a=\dfrac{\lambda_N\mu_2}{1-\mu_2}$, and we obtain $\g^{(1)}(N,\lambda,a,b)$. \end{proof}

\begin{prop}\label{prop8}
Among the examples of Theorem \ref{theo5}, the Com-PreLie bialgebras (or equivalently the Zinbiel-PreLie bialgebras) are 
$g^{(4)}(0)$ and $\g^{(1)}(1,a,1)$, with $a\neq 0$. 
\end{prop}

\begin{proof}  Note that $\g^{(1)}(1,0,1)=\g^{(4)}(0)$. 
Let us first prove that $\g(1,a,1)$ is a Zinbiel-PreLie bialgebra for any $a\in \K$.
By Remark \ref{rk1.1}, first point, it is enough to consider $\g(1,1,1)$. We put
\[A=\{x\in \g(1,1,1)\mid \forall y\in \g(1,1,1),\:\Delta(x\bullet y)=x^{(1)}\otimes x^{(2)}\bullet y+x^{(1)}\bullet y^{(1)}\otimes x^{(2)}y^{(2)}\}.\]
Firstly, $1\in A$: for any $y\in \g(1,1,1)$,
\begin{align*}
\Delta(1\bullet y)&=0=1\otimes 1\bullet y+1\bullet y^{(1)}\otimes 1y^{(2)}.
\end{align*}
Let $x_1,x_2\in A$. For any $y\in \g(1,1,1)$, by the Leibniz identity,
\begin{align*}
\Delta((x_1x_2)\bullet y)&=\Delta(x_1\bullet y)\Delta(x_2)+\Delta(x_1)\Delta(x_2\bullet y)\\
&=x_1^{(1)}x_2^{(1)}\otimes (x_1^{(2)}\bullet y)x_2^{(2)}+(x_1^{(1)}\bullet y^{(1)})x_2^{(1)}\otimes x_1^{(2)}x_2^{(2)}y^{(2)}\\
&+x_1^{(1)}x_2^{(1)}\otimes x_1^{(2)}(x_2^{(2)}\bullet y)+x_1^{(1)}(x_2^{(1)}\bullet y^{(1)})\otimes x_1^{(2)}x_2^{(2)}y^{(2)}\\
&=x_1^{(1)}x_2^{(1)}\otimes (x_1^{(2)}x_2^{(2)})\bullet y+(x_1^{(1)}x_2^{(1)})\bullet y^{(1)}\otimes x_1^{(2)}x_2^{(2)}y^{(2)}\\
&=(x_1x_2)^{(1)}\otimes (x_1x_2)^{(2)}\bullet y+(x_1x_2)^{(1)}bullet y^{(1)}\otimes (x_1x_2)^{(2)}y^{(2)}.
\end{align*}
So $x_1x_2\in A$: $A$ is a subalgebra of $\K[X]$. Hence, it is enough to prove that $X\in A$. Let $n\in \N$, let us consider $y=X^n$.
\begin{align*}
\Delta(X\bullet y)&=\frac{1}{1+n}\Delta(X^{n+1})\\
&=\sum_{k=0}^{n+1}\frac{n!}{k!(n+1-k)!} X^k\otimes X^{n+1-k};\\
X^{(1)}\bullet y^{(1)}\otimes X^{(2)}y^{(2)}&=X\bullet y^{(1)}\otimes y^{(2)}+0\\
&=\sum_{k=0}^n\frac{n!}{(k+1)!(n-k)}X^{k+1}\otimes X^{n-k}\\
&=\sum_{k=1}^{n+1}\frac{n!}{k!(n+1-k)!}X^k\otimes X^{n+1-k};\\
X^{(1)}\otimes X^{(2)}\bullet y&=1\otimes X\bullet y+0\\
&=\frac{n!}{0!(n+1-0)!}X^0\otimes X^{n+1-0}.
\end{align*}
This proves that $X\in A$, so $\g(1,1,1)$ is a Zinbiel-PreLie bialgebra.\\

Let $\g$ be one of the examples of Theorem \ref{theo5}.  Firstly,
\begin{align*}
\Delta(X\bullet X)&=X\otimes 1\bullet X+1\otimes X\bullet X\\
&+X\bullet X\otimes 1+X\bullet 1\otimes X+1\bullet X\otimes X+1\bullet 1\otimes X^2\\
\lambda_1(1\otimes X^2+2X\otimes X+X^2\otimes 1)&=\lambda_11\otimes X^2+\lambda X \otimes X+\lambda_1 X^2\otimes 1.
\end{align*}
This gives $\lambda_0=2\lambda_1$. In particular, if $\g=\g^{(4)}(\lambda)$, then $\lambda=2\lambda_1=0$: this is $\g^{(4)}(0)$.
In the other cases, $N$ exists. By definition of $N$, $X\bullet X^k=0$ if $1\leq k \leq N-1$. We obtain that
\begin{align*}
\Delta(X\bullet X^N)&=1\otimes X\bullet X^N+X \otimes 1\bullet X^N
+\sum_{k=0}^N \binom{N}{k} (X\bullet X^k \otimes X^{N-k}+1\bullet X^k \otimes X^{n-k+1})\\
\lambda_N \Delta(X^{N+1})&=1\otimes X\bullet X^N+\lambda X \otimes X^N+1\otimes X\bullet X^N.
\end{align*}
If $\lambda=0$, we obtain that $X^{N+1}$ is primitive, as $\lambda_N=0$, so $N+1=1$: absurd, $N\geq 1$.
So $\lambda \neq 0$,. The cocommutativity of $\Delta$ implies that $N=1$.
\begin{align*}
\Delta(X\bullet X^2)&=\lambda_2(X^3\otimes 1+3X^2\otimes X+3X\otimes X^2+1\otimes X^3)\\
&=1\otimes X\bullet X^2+2\lambda_1 X^2\otimes X+\lambda_0X\otimes X^2+1\otimes X\bullet X^2.
\end{align*}
Hence, $3\lambda_2=2\lambda_1$. 
\begin{itemize}
\item If $\g=\g^{(3)}(1,\lambda,\mu)$, we obtain $3\mu=2\mu$, so $\mu=0$: contradiction.
\item If $\g=\g^{(2)}(1,\lambda,\mu)$, we obtain $0=2\mu$, so $\mu=0$: contradiction.
\end{itemize}
So $\g=\g^{(1)}(1,\lambda,a,b)$. We obtain that
\[3\frac{a}{2+b}=2\frac{a}{1+b},\]
so $b=1$. Then $\lambda_0=2\lambda_1=\dfrac{2a}{2}=a=\dfrac{a}{b}$, so $\g=\g^{(1)}(1,a,1)$.\end{proof}

\subsection{Underlying Lie algebras}

We aim in this paragraph to describe the underlying Lie algebras of the pre-Lie algebras of Theorem \ref{theo5}. 
Let us first recall the construction of of the semi-direct sum of two Lie algebras. Let $\g$, $\mathfrak{h}$ be two Lie algebras 
and let $\tau:\mathfrak{h}\longrightarrow \der(\g)^{op}$
be a Lie algebra morphism, where $\der(\mathfrak{h})^{op}$ is the opposite of the Lie algebra of derivations of the Lie algebra $\mathfrak{h}$. 
Then $\g\oplus \h$ is given a Lie bracket by the following: if $x,x'\in \g$, $y,y'\in\mathfrak{h}$,
\begin{align*}
&\forall x,x'\in \g,\: \forall y,y'\in \mathfrak{h},&[x+y,x'+y']&=[x,x']_\g-\tau(y)(x')+\tau(y')(x)+[y,y']_{\mathfrak{h}}.
\end{align*}
This Lie algebra is denoted by $\g \oplus_\tau h$. Here are the examples we shall use in the sequel:
\begin{enumerate}
\item Let $\g$ be a graded pre-Lie algebra. Then the abelian Lie algebra $\K$ acts on $\g$ by derivation:
if $x\in \g$ is homogeneous of degree $n$, then $\tau(1)(x)=n x$. The associated semi-direct sum is denoted by $\g\oplus_{\deg} \K$.
\item Let $\g$ be a Lie algebra and let $(m,\cdot)$ be a right $\g$-module.
Considering $m$ as an abelian Lie algebra, we obtain a semi-direct product $m\oplus_\tau g$. For any $x,x'\in m$, $y,y'\in \g$,
\[[x+y,x'+y']=x\cdot y'-x'\cdot y+[y,y'].\]
\end{enumerate}

\begin{notation}
We shall use the Faà di Bruno Lie algebra $\g_{FdB}$: as a vector space it has a basis $(e_i)_{i\geq 1}$, and its Lie bracket is given by
\begin{align*}
&\forall k,l\geq 1,& [e_k,e_l]&=(k-l)e_{k+l}.
\end{align*}
This is the Lie algebra of the group of formal diffeomorphisms $\{x+a_1x^2+\ldots\}\subseteq \K[[x]]$, with the composition of formal series.
For any $\lambda \in \K$, the right $\g_{FdB}$-module $V_\lambda$ has a basis $(f_k)_{k\in \N}$ and
\begin{align*}
&\forall k,l\geq 1,&f_k\cdot e_l=(k+\lambda) f_{k+l}.
\end{align*}
\end{notation}

Any $\g$ described in Theorem \ref{theo5} can be decomposed into a semi-direct sum $\g_+\oplus \g_0$,
where $\g_0=\vect(1)$ and $\g_+=\vect(X^k,k\geq 1)$. The action of $\g_0$ over $\g_+$ is given by the product $\bullet$.
As a consequence, if $\lambda=0$, this is a trivial action and $\g$ is isomorphic to $\g_+\oplus \K$.
Otherwise, $\g$ is isomorphic to $\g_+\oplus_{\deg}\K$. Let us now describe $\g_+$.

\begin{prop}\label{prop9}
Let $\g$ be one of the Com-PreLie algebras of Theorem \ref{theo5} and let $\g_+$ its augmentation ideal.
\begin{enumerate}
\item If $\g=\g^{(1)}(N,\lambda,a,b)$ or $\g^{(3)}(N,\lambda,\mu)$ then, as a Lie algebra,
\[\g \approx \left(V_{\frac{1}{N}}\oplus \ldots \oplus V_{\frac{N-1}{N}}\right)\oplus_\tau \g_{FdB}.\]
\item If $\g=\g^{(2)}(N,\lambda,\mu)$, let us put $\g_1=\vect(f_k,1\leq k\neq N)$. Its is an abelian Lie algebra. Let $\tau$ be the action of $\K$ on $\g_1$
given by $k\cdot 1=f_{k+N}$ for any $k$. Then, as a Lie algebra,
\[\g_+\approx \g_1\oplus_\tau \g_2.\]
\item If $\g=\g^{(4)}(\lambda)$, then $\g_+$ is abelian.
\end{enumerate}\end{prop}

\begin{proof} The cases 2 and 3 are immediate. Let us consider the case $\g=\g^{(1)}(N,\lambda,a,b)$. 
We put $\g_0=\vect(X^{Nk},k\geq 1)$ and for any $i\in [N-1]$, $\g_i=\vect(X^{Nk+i},k\geq 1)$. For any $k\geq 1$, we put $e_k=\frac{k+b}{Na}X^k$. Then $(e_k)_{k\geq 1}$ is a basis of $\g_0$ and, for any $k,l\geq 1$,
\begin{align*}
[e_k,e_l]&=\frac{(k+b)(l+b)}{N^2a^2}(X^{kN}\bullet X^{lN}-X^{lN}\bullet X^{kN})\\
&=\frac{(k+b)(l+b)}{N^2a^2}a\left(\frac{kN}{l+b}-\frac{lN}{k+b}\right)X^{(k+l)N}\\
&=(k-l)\frac{k+l+b}{Na}X^{(k+l)N}\\
&=(k-l)e_{k+l}.
\end{align*}
So $\g_0$ is isomorphic to $\g_{FdB}$. By definition of the pre-Lie product, $\g_1\oplus\ldots \oplus \g_{N-1}$ is an abelian Lie algebra.
Moreover, if $i\in [N-1]$, $k\in \N$, $l\geq 1$,
\begin{align*}
[X^{kN+i},e_l]&=X^{kN+i}\bullet e_l+0\\
&=\frac{kN+i}{N}X^{(k+l)N+i}\\
&=\left(k+\frac{i}{N}\right)X^{(k+l)N+i}.
\end{align*}
So $\g_i$ is a right $\g_0$-module, isomorphic to $V_{\frac{i}{N}}$. The result follows. The proof for $\g^{(3)}(N,\lambda,\mu)$ is similar. \end{proof}

\section{Cocommutative Com-PreLie bialgebras}

We now prove the following theorem:

\begin{theo}\label{theo9}
Let $A$ be a connected, cocommutative Com-PreLie bialgebra. Then one of the following assertions holds:
\begin{enumerate}
\item There exist a linear form $f:\prim(A)\longrightarrow \K$ and $\lambda \in \K$, such that $A$ is isomorphic to $S(V,f,\lambda)$.
\item There exists $a\in \K$ such that $A$ is isomorphic to $\g^{(1)}(1,a,1)$.
\end{enumerate}\end{theo}

Firstly, observe that $A$ is a cocommutative, commutative, connected Hopf algebra: by the Cartier-Quillen-Milnor-Moore theorem,
it is isomorphic to the enveloping Hopf algebra of an abelian Lie algebra, so is isomorphic to $S(V)$ as a Hopf algebra,
where $V=\prim(A)$. If $V=(0)$, the first point holds trivially. 

\subsection{First case}

We assume in this paragraph that $V$ is at least $2$-dimensional.

\begin{lemma}
Let $A$ be a connected, cocommutative Com-PreLie algebra, such that the dimension of $\prim(A))$ is at least $2$. Then $f_A=0$,
and there exists a map $F:A\longrightarrow A$, such that:
\begin{enumerate}
\item For any $x,y \in A_+$, $x\bullet y=F(x\otimes y')y''+F(x\otimes 1)y$, with Sweedler's notation $\Delta(y)=y\otimes 1+1\otimes y+y'\otimes y''$.
\item For any $x_1,x_2 \in A$, $F(x_1x_2\otimes y)=F(x_1\otimes y)x_2+x_1F(x_2\otimes y)$.
\item $F(\prim(A)\otimes A)\subseteq \K$.
\end{enumerate}\end{lemma}

\begin{proof} We assume that $A=S(V)$ as a bialgebra, with its usual product and coproduct $\Delta$, and that $\dim(V)\geq 2$. Let $x,y\in V$. Then
\[\Delta(x\bullet y)=x\bullet y\otimes 1+1\otimes x\bullet y+f_A(x)\otimes y.\]
As $A$ is cocommutative, for any $x,y\in V$, $f_A(x)$ and $y$ are colinear. 
As $\dim(V)\geq 2$, necessarily $f_A=0$. \\

We now construct linear maps $F_i:V\otimes S^i(V)\longrightarrow \K$, such that for any $k\in \N$, with
\[F^{(k)}=\displaystyle \bigoplus_{i=0}^k F_i:\bigoplus_{i=0}^k V\otimes S^i(V)\longrightarrow \K,\]
for any $x\in V$, $y\in S^k(V)$,
\[x\bullet y=F^{(k)}(x\otimes y')\otimes y''+F^{(k)}(x\otimes 1)y.\] 
We proceed by induction on $k$. Let us first construct $F^{(0)}$. Let $x,y\in V$.
\[\Delta(x\bullet y^2)=1\otimes x\bullet y^2+x\bullet y^2\otimes 1+2x\bullet y\otimes y.\]
As $\Delta$ is cocommutative, $x\bullet y$ and $y$ are colinear, so there exists a linear map $g:V\longrightarrow \K$
such that $x\bullet y=g(x)y$. We the take $F^{(0)}(x\otimes 1)=g(x)$. For any $x,y\in V$, $x\bullet y=F(x\otimes 1)y$,
so the result holds for $k=0$. \\

Let us assume that $F^{(0)},\ldots,F^{(k-2)}$ are constructed for $k\geq 2$. Let $x,y_1,\ldots,y_k\in V$.
For any $I\subseteq [k]=\{1,\ldots,k\}$, we put $\displaystyle y_I=\prod_{i\in I} y_i$. Then
\[\bftDelta(y_1\ldots y_k)=\sum_{I\sqcup J=[k], I,J\neq \emptyset} y_I\otimes y_J,\]
and
\begin{align*}
\Delta(x\bullet y_1\ldots y_k)&=1\otimes x\bullet y_1\ldots y_k+x\bullet y_1\ldots y_k\otimes 1+\sum_{[k]=I\sqcup J,J\neq \emptyset} x\bullet y_I\otimes y_J\\
&=1\otimes x\bullet y_1\ldots y_k+x\bullet y_1\ldots y_k\otimes 1+\sum_{I\sqcup J\sqcup K=[k], J,K\neq \emptyset}
F^{(k-2)}(x \otimes y_I)\otimes y_J \otimes y_K. 
\end{align*}
 We put
 \[P(x,y_1\ldots y_k)=x\bullet y_1\ldots y_k-\sum_{I\sqcup J=[k], |J|\geq 2} F^{(k-2)}(x\otimes y_I)y_J.\]
 The preceding computation show that $P(x,y_1\ldots,y_k)$ is primitive, so belongs to $V$. Let $y_{k+1}\in V$. 
\begin{align*}
\bftDelta(x\bullet y_1\ldots y_{k+1})&=\sum_{I\sqcup J\sqcup K=[k+1],K\neq \emptyset,|J|\geq 2}
\underbrace{F^{(k-2)}(x\otimes y_I)y_J}_{\in S_{\geq 2}(V)}\otimes y_K\\
&+P(x,y_1\ldots y_k)\otimes y_{k+1}+\sum_{i=1}^kP(x, y_1\ldots y_{i-1}y_{i+1}\ldots y_{k+1})\otimes y_i.
\end{align*}
By cocommutativity, considering the canonical projection on $V\otimes V$, we deduce that $P(x, y_1\ldots y_k)\in \vect(y_1,\ldots, y_k,y_{k+1})$
for any nonzero $y_{k+1}\in V$. In particular, for $y_1=y_{k+1}$, $P(x\otimes y_1\ldots y_k)\in \vect(y_1,\ldots,y_k)$. 
By multilinearity, there exist $F'_1,\ldots,F'_k\in (V\otimes S_{k-1}(V))^*$, such that for any $x,y_1,\ldots,y_k \in V$,
\[P(x,y_1\ldots y_k)=F'_1(x\otimes y_2\ldots y_k)y_1+\ldots+F'_k(x\otimes y_1\ldots y_{k-1})y_k.\]
By symmetry in $y_1,\ldots,y_k$, $F'_1=\ldots=F'_k=F_{k-1}$. Then
\begin{align*}
x\bullet y_1\ldots y_k&=\sum_{I\sqcup J=[k],|j|\geq 2} F^{(k-2)}(x\otimes y_I)y_J+\sum_{I\sqcup J=[k],|J|=1}F_{k-1}(x\otimes y_I)y_J\\
&=\sum_{I\sqcup J=[k],|j|\geq 1} F^{(k-1)}(x\otimes y_I)y_J\\
&=F^{(k-1)}(x\otimes (y_1\ldots y_k)')(y_1\ldots y_k)''+F(x\otimes 1)y_1\ldots y_k.
\end{align*}
We finally defined a map $F:V\otimes S(V)\longrightarrow K$, such that for any $x\in V$, $b\in S_+(V)$,
\[x\bullet b=F(x\otimes b')b''+F(x\otimes 1)b.\]
We extend $F$ in a map from $S(V)\otimes S(V)\longrightarrow S(V)$ by $F(1\otimes b)=0$
and, for any $x_1,\ldots,x_k \in V$,
\[F(x_1\ldots x_k\otimes b)=\sum_{i=1}^k x_1\ldots x_{i-1}F(x_1\otimes b)x_{i+1}\ldots x_k.\]
This map $F$ satisfies points 2 and 3. We consider
\[B=\{a\in A\mid \forall  b\in S_+(V), a\bullet b=F(a\otimes b')b''+F(a\otimes 1)b\}.\]
As $1\bullet b=0$ for any $b\in S(V)$, $1\in B$. By construction of $F$, $V\subseteq B$.
Let $a_1,a_2\in B$. For any $b\in S_+(V)$,
\begin{align*}
a_1a_2\bullet b&=(a_1\bullet b)a_2+a_1(a_2\bullet b)\\
&=F(a_1\otimes b')a_2b''+a_1F(a_2\bullet b')b''+F(a_1\otimes 1)a_2b+a_1F(a_2\otimes 1)b\\
&=F(a_1a_2\otimes b')b''+F(a_1a_2\otimes 1)b. 
\end{align*}
So $a_1a_2\in B$. We obtain that $B$ is a subalgebra of $S(V)$ containing $V$, so is equal to $S(V)$: $F$ satisfies the first point. \end{proof}

\begin{remark} \begin{enumerate}
\item In this case, for any primitive element $v$, the $1$-cocycle of the coalgebra $A$ defined by $L(x)=a\bullet x$
is the coboundary associated to the linear form sending $x$ to $-F(a\otimes x)$.
\item In particular, the pre-Lie product of two elements $x,y$ of $\prim(A)$ si given by
\[x\bullet y=F(x\otimes 1)y.\]
\end{enumerate}\end{remark}

\begin{lemma}
With the preceding hypothesis, let us assume that $F(x\otimes 1)=0$ for all $x\in \prim(A)$. Then $\bullet=0$.
\end{lemma}

\begin{proof} We assume that $A=S(V)$ as a bialgebra. Note that for any $a,b\in S_+(V)$,
\[\bftDelta(a\bullet b)=a\bullet b'\otimes b''+a'\bullet b'\otimes a''b''+a'\bullet b\otimes a''+a'\otimes a''\bullet b.\]
Let us prove the following assertion by induction on $N$: for any $k<N$, for any $x,y_1,\ldots,y_k \in V$,
$x\bullet y_1\ldots y_k=0$. 
By hypothesis, this is true for $N=1$. Let us assume the result at a certain rank $N\geq 2$.
Let us choose $x,y_1,\ldots,y_N\in V$. Then, by the condition on $N$,
\[\bftDelta(x\bullet y_1\ldots y_N)=0+0+0+0=0.\]
So $x\bullet y_1\ldots y_N$ is primitive. 

Up to a factorization, we can write any $x \bullet y_1\ldots y_N$ as a linear span of terms of the form
$z_1\bullet z_1^{\beta_1}\ldots z_n^{\beta_n}$, with $z_1,\ldots,z_n$ linearly independent, $\beta_1,\ldots,\beta_n\in \N$,
with $\beta_1+\ldots+\beta_n=N$. If $n=1$, as $\dim(V)\geq 2$ we can choose any $z_2$ linearly independent with $z_1$ and take $\beta_2=0$.
It is now enough to consider $z_1\bullet z_1^{\beta_1}\ldots z_n^{\beta_n}$, with $n \geq 2$, $z_1,\ldots,z_n$ linearly independent, 
$\beta_1,\ldots,\beta_n\in \N$, $\beta_1+\ldots+\beta_n=N$. Let $\alpha_1,\ldots,\alpha_n\in \N$, such that $\alpha_1+\ldots+\alpha_n=N+1$.
\begin{align*}
 \bftDelta(z_1\bullet z_1^{\alpha_1}\ldots z_n^{\alpha_n})&=\sum_{i=1}^n \alpha_i z_1\bullet z_1^{\alpha_1}\ldots z_i^{\alpha_i-1}\ldots
 z_n^{\alpha_n}\otimes z_i,\\ \\
\bftDelta\left(\frac{z_1^2}{^2}\bullet z_1^{\alpha_1}\ldots z_n^{\alpha_n}\right)&=
\sum_{i=1}^n \alpha_i (z_1\bullet z_1^{\alpha_1}\ldots z_i^{\alpha_i-1}\ldots z_n^{\alpha_n})z_1\otimes z_i\\
&+\sum_{i=1}^n \alpha_i z_1\bullet z_1^{\alpha_1}\ldots z_i^{\alpha_i-1}\ldots z_n^{\alpha_n}\otimes z_1z_i\\
&+z_1\bullet z_1^{\alpha_1}\ldots z_n^{\alpha_n}\otimes z_1+z_1\otimes z_1\bullet z_1^{\alpha_1}\ldots z_n^{\alpha_n},\\ \\
(\bftDelta \otimes \id)\circ \bftDelta\left(\frac{z_1^2}{^2}\bullet z_1^{\alpha_1}\ldots z_n^{\alpha_n}\right)&=
\sum_{i=1}^n \alpha_i z_1\bullet z_1^{\alpha_1}\ldots z_i^{\alpha_i-1}\ldots z_n^{\alpha_n}\otimes z_1 \otimes z_i\\
&+\sum_{i=1}^n \alpha_i z_1\otimes z_1\bullet z_1^{\alpha_1}\ldots z_i^{\alpha_i-1}\ldots z_n^{\alpha_n}\otimes z_i\\
&+\sum_i \alpha_iz_1\bullet z_1^{\alpha_1}\ldots z_i^{\alpha_i-1}\ldots z_n^{\alpha_n}\otimes z_i\otimes z_1.
\end{align*}
The cocommutativity implies that for any $1\leq i\leq n$, $\alpha_i z_1\bullet z_1^{\alpha_1}\ldots z_i^{\alpha_i-1}\ldots z_n^{\alpha_n}$
and $z_i$ are colinear. We first choose $\alpha_1=\beta_1+1$, $\alpha_i=\beta_i$ for any $i\geq 2$, and we obtain for $i=1$ that
$z_1\bullet z_1^{\beta_1}\ldots z_n^{\beta_n} \in \vect(z_1)$. We then choose $\alpha_n=\beta_n+1$ and $\alpha_i=\beta_i$ for any $i\leq n-1$,
and we obtain for $i=n$ that $z_1\bullet z_1^{\beta_1}\ldots z_n^{\beta_n} \in \vect(z_n)$. Finally, as $n\geq 2$,
$z_1\bullet z_1^{\beta_1}\ldots z_n^{\beta_n} \in \vect(z_1)\cap Vec(z_2)=(0)$: the hypothesis is true at rank $N$.\\

We proved that for any $x\in V$, for any $b\in S(V)$, $x\bullet b=0$. By the derivation property of $\bullet$, as $V$ generates $S(V)$,
for any $a,b \in S(V)$, $a\bullet b=0$. \end{proof}

\begin{lemma} 
Under the preceding hypothesis, Let us assume that $F(\prim(A)\otimes \K)\neq (0)$. Then $A$ is isomorphic to a certain $S(\prim(A),f,\lambda)$,
with $f(x)=F(x\otimes 1)$ for any $x\in V$.
\end{lemma}

\begin{proof}  We assume that $A=S(V)$ as a bialgebra. Let $a,b,c\in S_+(V)$. Then
\begin{align*}
\bftDelta([a,b])&=a'\otimes a''\bullet b+a\bullet b'\otimes b''+a'\bullet b\otimes a''\\
&-b'\otimes b''\bullet a-b\bullet a'\otimes a''-b'\otimes a\otimes b''+[a',b']\otimes a''b'',
\end{align*}
where $[-,-]$ is the Lie bracket associated to $\bullet$. Hence,
\begin{align*}
(a\bullet b)\bullet c&=F(a\otimes 1)b\bullet c+F(a\otimes b')b''\bullet c\\
&=F(a\otimes 1)F(b\otimes 1)c+F(a\otimes 1)F(b\otimes c')c''\\
&+F(a\otimes b')F(b''\otimes 1)c+F(a\otimes b')F(b''\otimes c')c'',
\end{align*}
whereas
\begin{align*}
(a\bullet c)\bullet b&=F(a\otimes 1)F(c\otimes 1)b+F(a\otimes 1)F(c\otimes b')b''\\
&+F(a\otimes c')F(c''\otimes 1)b+F(a\otimes c')F(c''\otimes b')b'',
\end{align*}
and
\begin{align*}
a\bullet [b,c]&=F(a\otimes 1)F(b\otimes 1)c+F(a\otimes 1)F(b\otimes c')c''-F(a\otimes 1)F(c\otimes 1)b\\
&-F(a\otimes 1)F(c\otimes b')b''+F(a\otimes b')F(b''\otimes 1)c+F(a\otimes b')F(b''\otimes c')c''\\
&-F(a\otimes c')F(c''\otimes 1)b-F(a\otimes c')F(c''\otimes b')b''+F(a\otimes F(b\otimes 1)c')c''\\
&+F(a\otimes F(b\otimes c')c'')c'''-F(a\otimes F(c\otimes 1)b')b''-F(a\otimes F(c\otimes b')b'')b'''\\
&+F(a\otimes F(b'\otimes 1)c)b''+F(a\otimes F(b'\otimes c')c'')b''-F(a\otimes F(c'\otimes 1)b)c''\\
&+F(a\otimes F(c'\otimes b')b'')c''+F(a\otimes F(b'\otimes 1)c')b''c''+F(a\otimes F(b'\otimes c')c'')b''c'''\\
&-F(a\otimes F(c'\otimes 1)b')b''c'''-F(a\otimes F(c'\otimes b')b'')b'''c''.
\end{align*}
The pre-Lie identity implies that
\begin{align*}
0&=F(a\otimes F(b\otimes 1)c')c''+F(a\otimes F(b\otimes c')c'')c'''-F(a\otimes F(c\otimes 1)b')b''\\
&-F(a\otimes F(c\otimes b')b'')b'''+F(a\otimes F(b'\otimes 1)c)b''+F(a\otimes F(b'\otimes c')c'')b''\\
&-F(a\otimes F(c'\otimes 1)b)c''+F(a\otimes F(c'\otimes b')b'')c''+F(a\otimes F(b'\otimes 1)c')b''c''\\
&+F(a\otimes F(b'\otimes c')c'')b''c'''-F(a\otimes F(c'\otimes 1)b')b''c'''-F(a\otimes F(c'\otimes b')b'')b'''c''.
\end{align*}
For $a=x\in V$, $b=y\in V$, as $F(V\otimes S(V))\subset \K$, this simplifies to
\begin{equation}
\label{EQ1} F(x\otimes c')F(y\otimes 1)c''+F(y\otimes c')F(x\otimes c'')c'''=F(x\otimes F(c'\otimes 1)y)c''.
\end{equation}
Let $x_1,\ldots,x_k \in V$, linearly independent, $\alpha_1,\ldots,\alpha_k\in \N$,
with $\alpha_1+\ldots+\alpha_N\geq 1$. We take $c=x_1^{\alpha_1+1}\ldots x_k^{\alpha_k}$
and $d=x_1^{\alpha_1}\ldots x_k^{\alpha_k}$.
The coefficient of $x_1$ in (\ref{EQ1}), seen as an equality between two polynomials in $x_1,\ldots,x_k$, gives
\[(\alpha_1+1)(F(x\otimes d)F(y\otimes 1)+F(y\otimes d')F(x\otimes d''))=(\alpha_1+1)F(x\otimes F(d\otimes 1)y).\]
Hence, for any $x,y\in V$, for any $c\in S_+(V)$,
\begin{equation}
\label{EQ2} F(x\otimes c)F(y\otimes 1)+F(y\otimes c')F(x\otimes c'')=F(x\otimes F(c\otimes 1)y).
\end{equation}
We put $f(x)=F(x\otimes 1)$ for any $x\in V$. If $z_1,\ldots,z_k \in \ker(g)$, then
\[F(z_1\ldots z_k\otimes 1)=\sum_{i=1}^k z_1\ldots g(z_i)\ldots z_k=0.\]
Consequently, if $c\in S_+(\ker(f))\subseteq S_+(V)$, (\ref{EQ2}) gives
\[F(x\otimes c')F(y\otimes 1)+F(y\otimes c')F(x\otimes c'')=0.\]
Let us choose $y$ such that $F(y\otimes 1)=0$. An easy induction on the length of $c$ proves that for any $c\in S_+(\ker(g))$, 
$F(x\otimes c)=0$ for any $x\in V$. So there exist linear forms $g_k \in V^*$, such that for any $x,y_1,\ldots,y_k \in V$,
\[F(x\otimes y_1\ldots y_k)=g_k(x)f(y_1)\ldots f(y_k).\]
In particular, $h_0=f$. The pre-Lie product is then given by
\[x\bullet y_1\ldots y_k=\sum_{i=1}^{k-1}g_i(x)\sum_{1\leq j_1<\ldots<j_i\leq k} y_1\ldots f(y_{j_1})\ldots f(y_{j_i})\ldots y_k.\]

Let $x,y,z_1,\ldots,z_k\in V$.
\begin{align*}
x\bullet (y\bullet z_1\ldots z_k)&=x\bullet \sum_{i=0}^{k-1}g_i(y)\sum_{j_1,\ldots,j_i} z_i\ldots f(z_{j-1})\ldots f(z_{j_i})\ldots z_l\\
&=\sum_{i=0}^{k-1}g_{l-i-1}(x) g_i(x)\binom{l-1}{i}\sum_{j=1}^k f(z_1)\ldots f(z_{j-1})z_j f(z_{j+1})\ldots f(z_k)+S_{\geq 2}(V),\\
(x\bullet y)\bullet z_1\ldots z_k&=f(x)y\bullet z_1\ldots z_k\\
&=f(x)g_{k-1}(y)\sum_{j=1}^k f(z_1)\ldots f(z_{j-1})z_j f(z_{j+1})\ldots f(z_k)+S_{\geq 2}(V),\\
x\bullet (z_1\ldots z_k\bullet y)&=\sum_{i=1}^k f(y_i)x\bullet z_1\ldots z_{i-1}z_{i+1}\ldots z_ky\\
&=k g_{k-1}(x)f(z_1)\ldots f(z_k)y\\
&+(k-1)f(x)g_{k-1}(y)\sum_{j=1}^k f(z_1)\ldots f(z_{j-1})z_j f(z_{j+1})\ldots f(z_k)+S_{\geq 2}(V),\\
(x\bullet z_1\ldots z_k)\bullet y&=\sum_{i=0}^{k-1} g_i(x)\sum_{j_1,\ldots,j_i} z_1\ldots f(z_{j_1})\ldots f(z_{j_i})\ldots z_k \bullet y\\
&=kf(x)g_{k-1}(y)\sum_{j=1}^k f(z_1)\ldots f(z_{j-1})z_j f(z_{j+1})\ldots f(z_k)+S_{\geq 2}(V).
\end{align*}
Let us choose $z_1=\ldots=z_k$ such that $f(z)=1$. Then
\[\sum_{j=1}^k f(z_1)\ldots f(z_{j-1})z_j f(z_{j+1})\ldots f(z_k)=kz\neq 0.\]
The pre-Lie identity implies
\[f(x)g_{k-1}(y)+(k-1)g_{k-1}(x)f(y)-\sum_{i=0}^{k-1}g_i(y)g_{k-i-1}(x)\binom{k-1}{i}=0,\]
so, for any $l\geq 1$,
\begin{equation}
\label{EQ3} lg_l(x)f(y)=\sum_{i=1}^l g_i(y)g_{l_i}(x)\binom{l}{i}.
\end{equation}
Let us choose $x$ such that $f(x)=1$. Let us consider $y \in \ker(f)$, let us prove that $g_i(y)=0$ for any $i\in \N$.
As $g_0=0$, this is obvious for $i=0$. Let us assume the result at all ranks $<l$, with $l\geq 1$. Then (\ref{EQ3}) gives
\[0=\sum_{i=1}^{l-1} g_i(y)g_{l_i}(x)\binom{l}{i}+g_l(y)f(x)=g_l(y).\]
Consequently, for any $l\geq 1$, there exists a scalar $\lambda_l$ such that $g_l=\lambda f$. Equation (\ref{EQ3}, for $x,y$ such that $f(x)=f(y)=1$,
gives, for any $l\geq 1$,
\[l\lambda l=\sum_{i=1}^l \lambda_i \lambda_{l-i}\binom{l}{i}=\sum_{i=1}^{l-1}\lambda_i\lambda_{l-i}\binom{l}{i}+\lambda_l,\]
so, for any $l\geq 2$,
\[\lambda_l=\frac{1}{l-1}\sum_{i=1}^{l-1}\lambda_i\lambda_{l-i}\binom{l}{i}.\]
An easy induction then proves that $\lambda_l=l! \lambda_1^l$ for any $l\geq 1$. Putting $\lambda_1=\lambda$, for any $x,x_1,\ldots,x_n \in V$,
\[x\bullet x_1\ldots x_k=\sum_{I\subsetneq \{1,\ldots,k\}} |I|! \lambda^{|I|} f(x) \prod_{i\in I} f(x_i) \prod_{i\notin I} x_i.\]
This is the pre-Lie product of $S(V,f,\lambda)$. \end{proof}

\subsection{Second case}

We now assume that $V$ is one-dimensional. Then $S(V)$ and $\K[X]$ are isomorphic as bialgebras.
Let us describe all the pre-Lie products on $\K[X]$ making it a Com-PreLie bialgebra.

\begin{prop}
Let $\lambda,\mu \in \K$. We put
\begin{align*}
&\forall k,l\in \N,&X^k \bullet X^l&=\lambda k l! \sum_{i=k}^{k+l-1} \frac{\mu^{k+l-i-1}}{(i-k+1)!}X^i.
\end{align*}
Then $(\K[X],m,\prec,\Delta)$ is a Zinbiel-PreLie algebra, denoted by $\g'(\lambda,\mu)$.
\end{prop}

\begin{proof} This is obvious if $\lambda=0$. Let us assume that $\lambda \neq 0$. We consider a one-dimensional vector space $V$,
with basis $(x)$, $f\in V^*$ defined by $f(x)=\lambda$, in $S(V,f,\mu/\lambda)$, for any $k,l\in \N$,
\begin{align*}
x^k\bullet x^l&=k x^{k-1} x\bullet x^l\\
&=\sum_{j=0}^{l-1}\binom{l}{j}j!k \frac{\mu^j}{\lambda^j}\lambda^{j+1}x^{k+l-j-1}\\
&=\lambda k l! \sum_{j=0}^{l-1}\frac{\mu^j}{(l-j)!}x^{k°l-j-1}\\
&=\lambda k l! \sum_{i=k}^{k+l-1} \frac{\mu^{k+l-i-1}}{(i-k+1)!}x^i.
\end{align*}
Using the Hopf algebra morphism from $S(V)$ to $\K[X]$ sending $x$ to $X$, we obtain that $(\K[X],m,\bullet,\Delta)$ is a Com-PreLie bialgebra.
Let us now prove the Leibniz rule for the $\prec$ product. If $k,l\geq 1$ and $m\in \N$,
\begin{align*}
&(X^k \bullet X^m)\prec X^l+X^k\prec(X^l\bullet X^m)\\
&=\lambda m!\left(\sum_{i=k}^{k+m-1}\frac{\mu^{k+m-i-1}}{(i-k+1)!} \frac{ki}{i+l}X^{i+l}
+\sum_{i=l}^{l+m-1}\frac{\mu^{l+m-i-1}}{(i-l+1)!} \frac{kl}{k+i}X^{k+1}\right)\\
&=\lambda m!\sum_{i=k+l}^{k+l+m-1}\frac{\mu^{k+l+m-i-1}}{(i-k-l+1)!}\frac{k(i-l)+kl}{i}X^i\\
&=\lambda m!k\sum_{i=k+l}^{k+l+m-1}\frac{\mu^{k+l+m-i-1}}{(i-k-l+1)!}X^i\\
&=\frac{k}{k+l} \lambda m!(k+l)\sum_{i=k+l}^{k+l+m-1}\frac{\mu^{k+l+m-i-1}}{(i-k-l+1)!}X^i\\
&=\frac{k}{k+l}X^{k+l}\bullet X^m\\
&=(X^k\prec X^l)\bullet X^m.
\end{align*}
So $(\K[X],\prec,\bullet,\Delta)$ is a Zinbiel-PreLie bialgebra. \end{proof}

\begin{prop}\label{prop14}
Let $\bullet$ a pre-Lie product on $\K[X]$ such that $(\K[X],m,\bullet,\Delta)$ is a Com-PreLie bialgebra. Then $(\K[X],m,\bullet,\Delta)=\g^{(1)}(1,\lambda,1)$
for a certain $\lambda \in \K$, or $\g'(\lambda,\mu)$ for a certain $(\lambda,\mu)\in \K^2$.
\end{prop}

\begin{proof} Let $\pi:\K[X]\longrightarrow \K[X]$ be the canonical projection on $\vect(X)$, that is to say
\[\pi:\left\{\begin{array}{rcl}
\K[X]&\longrightarrow  \K[X]\\
X^k&\longrightarrow \delta_{k,1}X.
\end{array}\right.\]
For any $k\in \N$, we put $\pi(X\bullet X^k)=\lambda_k X$.\\

We shall use the map $\varpi=m\circ(\pi \otimes \id_{\K[X]}) \circ \Delta$. For any $k\in \N$,
\[\varpi(X^k)=m\circ(\pi \otimes \id_{\K[X]})\left(\sum_{i=0}^k \binom{k}{i}X^i \otimes X^{k-i}\right)=m(k X\otimes X^{k-1})=kX^k.\]

\textit{First step.} We fix $l\in \N$. For any $P,Q \in \K[X]$, $\varepsilon(P\bullet Q)=0$. Hence, we can write $\displaystyle X\bullet X^l=\sum_{i=1}^\infty a_iX^i$. 
\begin{align*}
\varpi(X\bullet X^l)&=\sum_{i=1}^\infty ia_iX^i\\
&=m\circ (\pi \otimes \id_{\K[X]})\circ \Delta(X\bullet X^l)\\
&=m\circ (\pi \otimes \id_{\K[X]})\left(1\otimes X\bullet X^l+\sum_{i=0}^l \binom{l}{i}X\bullet X^i\otimes X^{l-i}\right)\\
&=m\left(\sum_{i=0}^l \binom{l}{i}\lambda_i X\otimes X^{l-i}\right)\\
&=\sum_{i=0}^l \binom{l}{i}\lambda_i  X^{l-i+1}\\
&=\sum_{j=1}^{l+1}\binom{l}{l-j+1}\lambda_{l-j+1}X^j.
\end{align*}
Hence,
\[X\bullet X^l=\sum_{j=1}^{l+1}\binom{l}{l-j+1}\frac{\lambda_{l-j+1}}{j}X^j.\]
By the Leibniz identity, for any $k\in \N$, $X^k\bullet X^l=kX^{k-1}(X\bullet X^l)$, so for any $k,l\in \N$,
\[X^k \bullet X^l=\sum_{j=1}^{l+1} k\binom{l}{l-j+1} \frac{\lambda_{l-j+1}}{j} X^{j+k-1}.\]

\textit{Second step.} In particular, for any $k\in \N$, $X^k\bullet 1=k\lambda_0 X^k$, and $X\bullet X=\frac{\lambda_0}{2}X^2+\lambda_1X$. Hence,
\begin{align*}
X\bullet (X\bullet 1)-(X\bullet X)\bullet 1&=\frac{\lambda_0^2}{2}X^2+\lambda_0\lambda_1X-\frac{\lambda_0}{2}X^2\bullet 1-\lambda_1X\bullet 1\\
&=\frac{\lambda_0^2}{2}X^2+\lambda_0\lambda_1X-\lambda_0^2X^2-\lambda_0\lambda_1X\\
&=-\frac{\lambda_0^2}{2}X^2,
\end{align*}
whereas
\begin{align*}
X\bullet (1\bullet X)-(X\bullet 1)\bullet X&=0-\lambda_0X\bullet X=-\frac{\lambda_0^2}{2}X^2-\lambda_0\lambda_1X.
\end{align*}
By the pre-Lie identity, $\lambda_0\lambda_1=0$. We shall now study three sub-cases:
\begin{align*}
\mathbf{1.}\:&\lambda_0\neq 0.&
\mathbf{2.}\:&\lambda_0=0,\: \lambda_1=0.&
\mathbf{3.}\:&\lambda_0=0,\: \lambda_1\neq 0.
\end{align*}

\textit{Third step.} First sub-case: $\lambda_0\neq 0$. Let us prove that $\lambda_k=0$ for any $k\geq 1$ by induction on $k$. It is given by the second step if $k=1$.
Let us assume that $\lambda_1=\ldots=\lambda_{k-1}=0$.  Then $X\bullet X^k=\dfrac{\lambda_0}{k+1}X^{k+1}+\lambda_k X$, and
\begin{align*}
X\bullet (X^k \bullet 1)-(X\bullet X^k)\bullet 1&=k\lambda_0\left(\frac{\lambda_0}{k+1}X^{k+1}+\lambda_kX\right)
-\left(\frac{\lambda_0}{k+1}X^{k+1}+\lambda_kX\right)\bullet 1\\
&=\frac{k}{k+1}\lambda_0^2X^{k+1}+\lambda_0\lambda_k X-\lambda_0^2 X^{k+1}-\lambda_0\lambda_kX\\
&=\frac{-1}{k+1} \lambda_0^2X^{k+1};\\ 
X\bullet (1 \bullet X^k)-(X\bullet 1)\bullet X^k&=0-\lambda_0\left(\frac{\lambda_0}{k+1}X^{k+1}+\lambda_kX\right)\\
&=\frac{-1}{k+1} \lambda_0^2X^{k+1}-\lambda_0\lambda_kX.
\end{align*}
By the pre-Lie identity, $\lambda_0\lambda_k=0$. As $\lambda_0\neq 0$, $\lambda_k=0$.

Finally, we obtain, by the first step,  $X^k\bullet X^l=\lambda_0\frac{k}{l+1}X^{k+l}$ for any $k,l\in \N$.
So this is the pre-Lie product of $\K[X]_{\lambda_0}$.\\

\textit{Fourth step.} Second sub-case: $\lambda_0=\lambda_1=0$. Let us prove that $\lambda_k=0$ for any $k\in \N$. It is obvious if $k=0,1$.
Let us assume that $\lambda_0=\ldots=\lambda_{k-1}=0$, with $k\geq 2$.  Then $X^i\bullet X^j=0$ for any $j<k$, $i\in \N$. Hence,
\[X\bullet(X^{k+1}\bullet X^{k-1})=(X\bullet X^{k+1})\bullet X^{k-1}=(X\bullet X^{k-1})\bullet X^{k+1}=0.\]
By the pre-Lie identity, $X\bullet (X^{k-1}\bullet X^{k+1})=0$. Moreover,
\begin{align*}
X \bullet (X^{k-1}\bullet X^{k+1})&=X\bullet\left(\sum_{j=1}^{k-2}\binom{k+1}{k+2-j}(k-1) \frac{\lambda_{k+2-j}}{j}X^{k+2-j}\right)\\
&=X\bullet\left((k-1)\lambda_{k+1}X^{k-1}+(k+1)(k-1)\frac{\lambda_k}{2}X^k\right)\\
&=0+\frac{(k-1)(k+1)}{2}\lambda_k X\bullet X^k\\
&=\frac{(k-1)(k+1)}{2}\lambda_k\left(\sum_{j=1}^{k-1} \binom{k}{k+1-j} \frac{\lambda_{k+1-j}}{j}X^j \right)\\
&=\frac{(k-1)(k+1)}{2}\lambda_k^2 X+0.
\end{align*}
Hence, $\lambda_k=0$.

We finally obtain by the first step $X^k\bullet X^l=0$ for any $k,l\in \N$. So this is the trivial pre-Lie product of $\K[X]_0=\K[X]_{0,0}$.\\

\textit{Fifth step.} Last sub-case: $\lambda_0=0$, $\lambda_1\neq0$. Let us prove that 
$\lambda_k=\displaystyle \frac{k!}{2^{k-1}}\frac{\lambda_2^{k-1}}{\lambda_1^{k-2}}$ for any $k \geq 1$. It is obvious if $k=1$ or $k=2$. 
Let us assume the result at all rank $<k$, with $k\geq 2$. 
\begin{align*}
\pi((X\bullet X) \bullet X^k)&=\pi(\lambda_1 X \bullet X^k)=\lambda_1\lambda_kX,
\end{align*}
whereas
\begin{align*}
\pi(X\bullet (X\bullet X^k))&=\pi\left(\sum_{j=1}^k \binom{k}{k+1-j}\frac{\lambda_{k+1-j}}{j}X\bullet X^j \right)\\
&=\sum_{j=1}^k \binom{k}{k+1-j}\frac{\lambda_{k+1-j}\lambda_j}{j}X\\
&=\left(\lambda_k\lambda_1+\sum_{j=2}^{k-1}\frac{1}{j}\binom{k}{k+1-j} \frac{(k+1-j)!j!}{2^{k-j+j-1}} 
\frac{\lambda_2^{k-j+j-1}} {\lambda_1^{k-j-1+j-2}}+\frac{k}{k}\lambda_1\lambda_k\right)X\\
&=\left(2\lambda_1\lambda_k+\sum_{j=2}^{k-1}\frac{k!}{2^{k-1}}\frac{\lambda_2^{k-1}}{\lambda_1^{k-3}}\right)X\\
&=\left(2\lambda_1\lambda_k+(k-2)\frac{k!}{2^{k-1}}\frac{\lambda_2^{k-1}}{\lambda_1^{k-3}}\right)X
\end{align*}
and
\begin{align*}
\pi((X\bullet X^k)\bullet X)&=\sum_{j=1}^k \binom{k}{k+1-j}\frac{\lambda_{k+1-j}}{j}\pi(X^j \bullet X)\\
&=\sum_{j=1}^k \binom{k}{k+1-j}\frac{\lambda_{k+1-j}}{j}\pi(j\lambda_1 X^j)\\
&=\lambda_1\lambda_kX+0
\end{align*}
and finally
\begin{align*}
\pi(X\bullet(X^k\bullet X))&=k\lambda_1 \pi(X\bullet X^k)=k\lambda_1 \lambda_kX.
\end{align*}
By the pre-Lie identity,
\[\lambda_1\lambda_k-2\lambda_1\lambda_k-(k-2)\frac{k!}{2^{k-1}}\frac{\lambda_2^{k-1}}{\lambda_1^{k-3}}
=\lambda_1\lambda_k-k\lambda_1\lambda_k,\]
which gives, as $\lambda_1\neq 0$ and $k \geq 3$, $\displaystyle \lambda_k=\frac{k!}{2^{k-1}} \dfrac{\lambda_2^{k-1}}{\lambda_1^{k-2}}$.
Finally, the first step gives, for all $k,l \in \N$, with $\lambda=\lambda_1$ and $\mu=\dfrac{\lambda_2}{2\lambda_1}$,
\begin{align*}
X^k \bullet X^l&=\sum_{j=1}^{k+1} k\binom{l}{l+1-j} \frac{\lambda_{l+1-j}}{j}X^{j+k-1}\\
&=\sum_{j=1}^k k\frac{l!(l+1-j)!}{(l+1-j)! (j-1)!j 2^{l-j}} \frac{\lambda_2^{l-j}}{\lambda_1^{l-1-j}}X^{j+k-1}\\
&=\lambda k l! \sum_{j=1}^k \frac{\mu^{l-j}}{j!} X^{j+k-1}\\
&=\lambda k l! \sum_{i=k}^{k+l-1} \frac{\mu^{k+l-i-1}}{(i-k+1)!} X^i.
\end{align*}
So this is the pre-Lie product of $\g'(\lambda,\mu)$. \end{proof}

As $\g'(\lambda,\mu)$ is a special case of $S(V,f,\lambda)$, this ends the proof of Theorem \ref{theo9}.
Let us describe the underlying Lie algebra of $\g'(\lambda,\mu)$, in a similar way as Proposition \ref{prop9}.
If $\lambda=0$, it is abelian. Otherwise:

\begin{prop} \label{prop16}
Let $\g=\g'(\lambda,\mu)$, with $\lambda\neq 0$. As a Lie algebra,
\[\g\approx (\g_{FdB}\oplus_{\deg}\K)\oplus \K.\]
\end{prop}

\begin{proof} Let us denote by $\g_+$ the augmentation ideal of $\g$. Obviously, $\g=\g_+\oplus \K$.
For any $k\geq 1$, we put $\displaystyle Y_k=\frac{X^k}{\lambda}$. Then, for any $k,l\geq 1$,
\begin{align*}
Y_k\bullet Y_k&=k Y_{k+l-1}+\mbox{terms of smaller degree},\\
[Y_k,Y_l]&=(k-l)Y_{k+l-1}+\mbox{terms of smaller degree}.
\end{align*}
We put $f_1=Y_1$. For any $n\geq 1$, we put $V_n=\vect(Y_1,\ldots,Y_n)$. The matrix of the endomorphism of $V_n$ sending $x$ to $[x,f_1]$ is triangular,
with diagonal $(0,1,\ldots,n-1)$. Consequently, it is diagonalizable, and for any $n\geq 1$, there exists a unique $f_n$ of the form $f_n=Y_n+a_{n-1}Y_{n-1}+\ldots+a_1 Y_1$,
such that $[f_n,f_1]=(n-1)f_n$. We obtain a basis $(f_n)_{n\geq 1}$ of $\g_+$. For any $k,l\geq 1$,
\begin{align*}
[[f_k,f_l],f_1]&=[[f_k,f_1],f_l]+[f_k,[f_l,f_1]]\\
&=(k-1+l-1)[f_k,f_l],
\end{align*}
so there exists a scalar $\lambda(k,l)$, such that $[f_k,f_l]=\lambda(k,l)f_{k+l-1}$. Moreover,
\begin{align*}
[f_k,f_l]&=[Y_k,Y_l]+\mbox{terms of smaller degree}\\
&=(k-l)Y_{k+l-1}+\mbox{terms of smaller degree}.
\end{align*}
So $[f_k,f_l]=(k-l)f_{k+l-1}$. Let $\g_1=\vect(f_k,k\geq 2)$ and $\g_2=\vect(f_1)$. Both are Lie subalgebras of $\g_+$ and, as a vector space, $\g=\g_1\oplus \g_2$.
We put $e_k=f_{k+1}$ for any $k\geq 1$. Then, for any $k,l\geq 1$,
\[[e_k,e_l]=[f_{k+1},f_{l+1}]=(k-1-l+1)f_{k+l+1}=(k-l)e_{k+l},\]
so $\g_1$ is isomorphic to $\g_{FdB}$. For any $k\geq 1$,
\[[e_k,f_1]=(k+1-1)f_{k+1}=ke_k,\]
so $\g_+\approx \g_{FdB}\oplus_{\deg} \K$. \end{proof}

\begin{cor}
Let $\g=S(V,f,\lambda)$ be a Com-PreLie algebra of Theorem \ref{theo5}. If $f$ is nonzero, then, as a Lie algebra,
\[\g\approx ( (\g_{FdB}\oplus_{\deg}\K)\oplus \K)\otimes S(\ker(f)),\]
where the Lie bracket is given by the following: for any $x,x'\in (\g_{FdB}\oplus_{\deg}\K)\oplus \K$ and $y,y'\in S(\ker(f))$,
\[ [x\otimes y,x'\otimes y']=[x,x']\otimes yy'.\]
\end{cor}

\begin{proof} Let $x_1\in V$, such that $f(x_1)=1$. Then, as an algebra, $\g\equiv \K[x_1]\otimes S(\ker(f))$. Moreover, $\K[x_1]$ is a Com-Prelie bialgebra,
isomorphic to $\g'(1,\lambda)$. Moreover, if $k,l \in \N$ and $y,z\in S(\ker(f))$, by definition of the pre-Lie product, 
\[[x_1^ky,x_1^lz]=((k-l)x_1^{k+l-1}+\mbox{terms of smaller degree})yz.\]
The end of the proof is similar to to the one of Proposition \ref{prop16}. \end{proof}

\section{Com-PreLie structures on group Hopf algebras}

Let $G$ be an abelian group. The group algebra $\K G$ is a Hopf algebra, which product is given by the bilinear extension of the product of $G$ and
\begin{align*}
&\forall g\in G,&\Delta(g)&=g\otimes g.
\end{align*}

\begin{lemma} \label{lem15}
Let $G$ be a group. For any $g,h\in G$, we put
\[P_{g,h}=\{x\in \K G\mid \Delta(x)=g\otimes x+x\otimes h\}.\]
Then $P_{g,h}=\vect(g-h)$.
\end{lemma}

\begin{proof} Firstly,
\begin{align*}
\Delta(g-h)&=g\otimes g-h\otimes h=g\otimes (g-h)+(g-h)\otimes h.
\end{align*}
Let $\displaystyle x=\sum_{k\in G} a_kk\in P_{g,h}$. Then
\begin{align*}
\Delta(x)&=\sum_{k\in G} a_k k\otimes k=\sum_{k\in G} a_k (g\otimes k+k\otimes h).
\end{align*}
Identifying in the basis $(k\otimes l)_{k,l\in G}$ of $\K G^{\otimes 2}$, we obtain that:
\begin{itemize}
\item If $k\notin \{g,h\}$, $a_k=0$.
\item If $g\neq h$, identifying the coefficients of $g\otimes h$, we obtain $0=a_g+a_h$. In this case, $x=a_g (g-h)$.
\item If $g=h$, identifying the coefficients of $g\otimes g$, we obtain $a_g=2a_g$, so $a_g=0$. In this case, $x=0=0(g-g)$.
\end{itemize}
So $P_{g,h}=\vect(g-h)$. \end{proof}

\subsection{General case}

\begin{theo} \label{theo16}
Let $G$ be an abelian group and $\bullet$ be a bilinear product on $\K G$. Then $(\K G,m,\bullet,\Delta)$ is a Com-PreLie bialgebra if, and only if,
there exists a family $(\lambda(g,h))_{g,h\in G}$ of scalars such that
\begin{align}
\nonumber &\forall g\in G,&\lambda(g,1)&=0.\\
&\forall g,h,k\in G,&\lambda(gh,k)&=\lambda(g,k)+\lambda(h,k).\\
\nonumber&\forall g,h,k\in G\mbox{ with }hk\neq 1,&&\lambda(g,hk)\lambda(h,k)-\lambda(g,h)\lambda(h,k)\\
&&&=\lambda(g,hk)\lambda(k,h)-\lambda(g,k)\lambda(k,h).\\
&\forall g,h\in G,&g\bullet h&=\lambda(g,h)(g-gh).
\end{align}\end{theo}

\begin{proof} Let us assume that $\bullet$ makes $\K G$ a Com-PreLie bialgebra. For any $g,h\in G$,
\begin{align*}
\Delta(g\bullet h)&=g\otimes g\bullet h+g\bullet h\otimes gh.
\end{align*}
By Lemma \ref{lem15}, there exists a scalar $\lambda(g,h)$ such that $g\bullet h=\lambda(g,h)(g-gh)$.\\

Let $\bullet$ be a product defined on $\K G$ by $g\bullet h=\lambda(g,h) (g-gh)$, for any $g,h\in G$.
For any $g\in G$, $g\bullet h=0$: we can assume that $\lambda(g,1)=0$.
For any $g,h\in G$,
\begin{align*}
\Delta(g\bullet h)&=g\otimes g\bullet h+g\bullet h\otimes gh,
\end{align*}
so $\bullet$ satisfies the compatibility with $\Delta$. Moreover,
\begin{align*}
&\mbox{$\bullet$ satisfies the Leibniz identity}\\
&\Longleftrightarrow \forall g,h,k\in G,\: (gh)\bullet k=(g\bullet k)h+g(h\bullet k)\\
&\Longleftrightarrow \forall g,h,k\in G,\: \lambda(gh,k)(gh-ghk)=\lambda(g,k)(gh-ghk)+\lambda(h,k)(gh-ghk)\\
&\Longleftrightarrow \forall g,h,k\in G,\: \lambda(gh,k)(gh-ghk)=(\lambda(g,k)+\lambda(h,k))(gh-ghk)\\
&\Longleftrightarrow \forall g,h\in G,k\in G\setminus \{1\},\: \lambda(gh,k)=\lambda(g,k)+\lambda(h,k).
\end{align*}
As $\lambda(g,1)=0$ for any $g\in G$, this identity is trivially satisfies if $k=1$. Hence,
\begin{align*}
\mbox{$\bullet$ satisfies the Leibniz identity}
&\Longleftrightarrow \forall g,h,k\in G,\: \lambda(gh,k)=\lambda(g,k)+\lambda(h,k).
\end{align*}
We now assume that $\bullet$ satisfies the Leibniz identity. Let $g,h,k\in G$.
\begin{align*}
(g\bullet h)\bullet k-g\bullet (h\bullet k)
&=(\lambda(g,h)\lambda(g,k)-\lambda(g,h)\lambda(h,k)+\lambda(g,hk\lambda(h,k))g\\
&-\lambda(g,h)\lambda(g,k)(gh+gk)+(\lambda(g,h)\lambda(g,k)\\
&+\lambda(g,h)\lambda(h,k)-\lambda(g,hk)\lambda(h,k))ghk.
\end{align*}
Hence,
\begin{align*}
&\mbox{$\bullet$ is pre-Lie}\\
&\Longleftrightarrow \forall g,h,k\in G,\: (g\bullet h)\bullet k-g\bullet (h\bullet k)=(g\bullet k)\bullet h-g\bullet (k\bullet h)\\
&\Longleftrightarrow \forall g,h,k\in G, \:(\lambda(g,hk)\lambda(h,k)-\lambda(g,h)\lambda(h,k))(g-ghk)\\
&\hspace{3cm} =(\lambda(g,hk)\lambda(k,h)-\lambda(g,k)\lambda(k,h))(g-ghk)\\
&\Longleftrightarrow \forall g,h,k\in G\mbox{ with }hk\neq 1, \:\lambda(g,hk)\lambda(h,k)-\lambda(g,h)\lambda(h,k)\\
&\hspace{5cm} =\lambda(g,hk)\lambda(k,h)-\lambda(g,k)\lambda(k,h).
\end{align*}
Consequently, $\bullet$ makes $\K G$ a Com-PreLie bialgebra if, and only if, the four conditions of Theorem \ref{theo16} are satisfied. \end{proof}

\begin{remark} \begin{enumerate}
\item For any $h\in G$, $\lambda(-,h):G\longrightarrow (\K,+)$ is a group morphism.
\item If $g\in G$ is an element of finite order $n$, then for any $h\in G$,
\begin{align*}
\lambda(g^n,h)&=n\lambda(g,h)=\lambda(1,h)=0.
\end{align*}
As the characteristic of $\K$ is zero, $\lambda(g,h)=0$. 
Consequently, if any element of $G$ is of finite order, then the only product making $\K G$ a Com-PreLie bialgebra is $0$.
\end{enumerate}\end{remark}

\begin{prop}\label{prop17}
Let $G$ be an abelian group. Let $\lambda:G\longrightarrow (\K,+)$ be a group morphism and let $g_0\in G$.
We define a product $\bullet$ on $\K G$ by
\begin{align*}
&\forall g,h\in G,&g\bullet h&=\lambda(g) \delta_{h,g_0}g(1-g_0).
\end{align*}
Then $(\K G,m,\bullet,\Delta)$ is a Com-PreLie bialgebra.
 \end{prop}

\begin{proof} First, note that if $g_0=1$, then $\bullet=0$. We now assume that $g_0\neq 1$.

Here, $\lambda(g,h)=\lambda(g) \delta_{h,g_0}$ for any $g,h\in G$. Let us prove that the conditions of Theorem \ref{theo16} are satisfied.
let $g,h,k\in G$. Then
\begin{align*}
\lambda(g,k)+\lambda(h,k)&=\delta_{k,g_0}(\lambda(g)+\lambda(h))=\delta_{k,g_0}\lambda(gh)=\lambda(gh,k).
\end{align*} 
As $\lambda(1)=0$ and $g_0\neq 1$,
\begin{align*}
\lambda(g,hk)\lambda(h,k)-\lambda(g,h)\lambda(h,k)&=\lambda(g)\lambda(h)\delta_{k,g_0}(\delta_{hk,g_0}-\delta_{h,g_0})\\
&=\begin{cases}
-\lambda(g)\lambda(h) \mbox{ if }h=k=g_0,\\
\lambda(g)\lambda(h)\mbox{ if }h=1\mbox{ and }k=g_0,\\
0\mbox{ otherwise};
\end{cases}\\
&=\begin{cases}
-\lambda(g)\lambda(h) \mbox{ if }h=k=g_0,\\
0\mbox{ otherwise};
\end{cases}\\
&=\lambda(g,hk)\lambda(k,h)-\lambda(g,k)\lambda(k,h).
\end{align*}
So $(\K G,m,\bullet,\Delta)$ is indeed a Com-PreLie bialgebra. \end{proof}

\subsection{Examples on $\Z$}

Our aim here is the classification of all Com-PreLie bialgebra structures on $\K \Z$.
In order to avoid confusion between the sum of $\Z$ and the sum of $\K \Z$, we identify $\K \Z$ with the Laurent polynomial algebra $\K[X,X^{-1}]$,
 with the coproduct defined by $\Delta(X)=X\otimes X$. 
 
 \begin{theo} \label{theo18}
 There are three families of products $\bullet$ on $\K[X,X^{-1}]$ making it a Com-PreLie bialgebra:
 \begin{enumerate}
 \item $\bullet=0$.
 \item There exist $k_0\in \Z$, nonzero, $a\in \K$, nonzero, such that
 \begin{align*}
& \forall k,l\in \Z,&X^k \bullet X^l&=a\delta_{l,k_0} k(X^k-X^{k+l}).
 \end{align*}
 \item There exist $\alpha,\beta \in \K\setminus \{0\}$, such that for any $n\neq -1$, $n\alpha-(n-1)\beta \neq 0$, and $N\geq 1$ such that
  \begin{align*}
& \forall k,l\in \Z,&X^k \bullet X^l&=\begin{cases}
\displaystyle \frac{\alpha\beta}{\left(\frac{l}{N}-1\right)\alpha-\left(\frac{l}{N}-2\right)\beta}k(X^k-X^{k+l})
\mbox{ if }N\mid l,\\
0\mbox{ otherwise}.
\end{cases}
 \end{align*}
 \end{enumerate} \end{theo}

\begin{proof} We shall use Theorem \ref{theo16}. Let $\bullet$ be a product on $\K[X,X^{-1}]$, making it a Com-PreLie bialgebra.
Then
\begin{align*}
&\forall k,l\in \Z,&X^k\bullet X^l&=\lambda(k,l)(X^k-X^{k+l}).
\end{align*}
Moreover, for any $l\in \Z$, $\lambda(-,l):(\Z,+)\longrightarrow(\K,+)$ is a group morphism so there exists a scalar $a_l$ such that 
for any $k\in \Z$, $\lambda(k,l)=a_lk$. The conditions of Theorem \ref{theo16} become:
\begin{itemize}
\item $a_0=0$.
\item For any $h,k\in \Z$, such that $h+k\neq 0$,
\begin{align}
\nonumber a_{h+k}a_kh-a_ka_hh&=a_{h+k}a_hk-a_ka_hk\\
\label{EQ8} \Longleftrightarrow a_{h+k}(a_kh-a_hk)&=a_ha_k(h-k).
\end{align}
\end{itemize}

Let us assume that there exists $n\geq 1$, such that $a_n\neq 0$. Let $N=\min\{n\geq 1, a_n\neq 0\}$.
Let us prove that for any $n\geq 1$, if $a_n\neq 0$, then $N\mid n$. Let us write $n=Nq+r$, with $0 \leq r\leq N-1$.
We proceed by induction on $q$. By definition of $N$, $a_0=a_1=\ldots=a_{N-1}=0$, so this proves the result for $q=0$.
Let us assume the result at rank $q-1$. If $r\neq 0$, then by the induction hypothesis, $a_{(q-1)N+r}=0$. Hence, by (\ref{EQ8})
for $h=(q-1)N+r$ and $k=N$,
\begin{align*}
a_{qN+r}a_N((q-1)N+r)&=0.
\end{align*}
As $a_N\neq 0$, $a_{qN+r}=0$.

Let us put $b_n=a_{Nn}$ for any $n\geq 1$. Then $b_1\neq 0$ and, for any $h,k\geq 1$,
\begin{align}
\label{EQ9} b_{h+k}(b_kh-b_hk)&=b_hb_k(h-k).
\end{align}

Let us assume that $b_2=0$. We prove that $b_n=0$ for any $n\geq 2$ by induction on $n$. This is obvious if $n=2$.
If $b_n=0$, with $n\geq 2$, by (\ref{EQ9}), with $h=n$ and $k=1$, $b_{n+1}b_1n=0$. As $b_1\neq 0$, $b_{n+1}=0$.

Let us assume that $b_2\neq 0$. Let us show that for any $n\geq 1$, $(n-1)b_1\neq (n-2)b_2$ and
\begin{align*}
b_n&=\frac{b_1b_2}{(n-1)b_1-(n-2)b_2}.
\end{align*}
This is obvious if $n=1$ or $n=2$, as $b_1,b_2\neq 1$. Let us assume the result at rank $n$, $n\geq 2$.
By (\ref{EQ9}) with $h=n$, $k=1$,
\begin{align*}
b_{n+1}\left(b_1n-\frac{b_1b_2}{(n-1)b_1-(n-2)b_2}\right)&=b_1\frac{b_1b_2}{(n-1)b_1-(n-2)b_2}(n-1)\\
b_{n+1}(n(n-1)b_1-(n-1)^2 b_2)&=b_1b_2(n-1)\\
b_{n+1}(nb_1-(n-1)b_2)=b_1b_2.
\end{align*}
As $b_1,b_2\neq 0$, $nb_1-(n-1)b_2\neq 0$, so
\begin{align*}
b_{n+1}=\frac{b_1b_2}{nb_1-(n-1)b_2}.
\end{align*}

We proved that there are three possibilities for $(a_n)_{n\geq 1}$:
\begin{enumerate}
\item For any $n\geq 1$, $a_n=0$.
\item There exists a unique $N\geq 1$ such that $a_N\neq 0$.
\item There exist $N\geq 1$, $\alpha,\beta \neq 0$ such that for any $n\in \N$, $n\alpha-(n-1)\beta\neq 0$ and
\begin{align*}
a_n&=\begin{cases}
\displaystyle \frac{\alpha \beta}{\left(\frac{n}{N}-1\right)\alpha-\left(\frac{n}{N}-2\right)\beta}\mbox{ if }N\mid n,\\
0\mbox{ otherwise.}
\end{cases}
\end{align*}
\end{enumerate}
Similarly, there are three possibilities for $(a_{-n})_{n\geq 1}$:
\begin{enumerate}
\item[1.'] For any $n\geq 1$, $a_{-n}=0$.
\item[2.'] There exists a unique $N'\geq 1$ such that $a_{-N'}\neq 0$.
\item[3.'] There exist $N'\geq 1$, $\alpha',\beta' \neq 0$ such that for any $n\in \N$, $n\alpha'-(n-1)\beta'\neq 0$ and
\begin{align*}
a_{-n}&=\begin{cases}
\displaystyle \frac{\alpha' \beta'}{\left(\frac{n}{N'}-1\right)\alpha'-\left(\frac{n}{N'}-2\right)\beta'}\mbox{ if }N'\mid n,\\
0\mbox{ otherwise.}
\end{cases}
\end{align*}
\end{enumerate}
If (2. or 3.) and (2.' or 3.') is satisfied, let us assume that $N\neq N'$. For example, we assume that $N'>N$. 
By (\ref{EQ8}) with $h=-N'$ and $k=N$, then
\begin{align*}
a_{N-N'}(-N'a_N-Na_{-N'})=a_{-N'}a_N(-N'-N)\neq 0.
\end{align*}
So $a_{N-N'}\neq 0$: this contradicts the definition of $N'=\min\{n\geq 1,a_{-n}\neq 0\}$. So $N=N'$. \\

Let us assume that ((1'. or 2') and 3.) is satisfied. By (\ref{EQ8}) for $h=3N$ and $k=-2N$,
\begin{align*}
a_N(2Na_{3N})=0.
\end{align*}
We obtain $a_N=0$, so $\alpha=0$: contradiction. So  ((1'. or 2') and 3.)  is impossible. Similarly, (1. or 2) and 3'.) is impossible. 
If (2. and 2.') is satisfied, by (\ref{EQ8}) for $h=N$ and $k=-2N$,
\begin{align*}
a_{-N}(2Na_N)=0.
\end{align*}
As $a_{-N}$ and $a_N$ are both nonzero, this is a contradiction. So (2. and 2.') is impossible.\\

It remains the following cases:

\begin{enumerate}
\item If (1. and 1.') is satisfied, then $\bullet=0$.
\item  If (1. and 2.') or (1.' and 2.) is satisfied, this is the second case of Theorem \ref{theo18}.
\item If (3. and 3.') is satisfied, by (\ref{EQ8}) with $h=2N$ and $k=-N$, $h=3N$ and $k=-2N$, we obtain that
\begin{align*}
\alpha'&=\frac{-\alpha \beta}{2\alpha-3\beta},&\beta'=\frac{-\alpha\beta}{3\alpha-4\beta}.
\end{align*}
Hence, if $n\geq 1$,
\begin{align*}
a_{-nN}&=\frac{\alpha\beta}{(-n-1)\alpha-(-n-2)\beta}.
\end{align*}
This is the third case of Theorem \ref{theo18}.
\end{enumerate}

It remains to prove that this three cases give indeed Com-PreLie bialgebras. It is obvious in the first case.
The second case is Proposition \ref{prop17}. The third case is left to the reader. \end{proof}

\begin{remark}
\begin{enumerate}
\item  In the third case, if $-\alpha+2\beta=0$, the formula for $X^k\bullet X^0$ is not well-defined: by convention, $X^k \bullet 1=0$.
\item Another to present the third case is through the change of variables $\gamma=\alpha-\beta$ and $\delta=-\alpha+2\beta$. Then, if for any $n\neq -1$, $\gamma(n+1)+\delta\neq 0$,
 \begin{align*}
& \forall k,l\in \Z,&X^k \bullet X^l&=\begin{cases}
\displaystyle \frac{(\gamma+\delta)(2\gamma+\delta)}{\frac{l}{N}\gamma+\delta}k(X^k-X^{k+l})
\mbox{ if }N\mid l,\\
0\mbox{ otherwise}.
\end{cases}
 \end{align*}
\end{enumerate}

\end{remark}

\section{Examples of non connected Com-PreLie bialgebras}

Let $H$ be a commutative and cocommutative Hopf algebra. If $\K$ is an algebraically closed field of characteristic zero, denoting by $G$
the group of group-like elements of $H$ and by $V$ the space of its primitive elements, then $H$ is isomorphic to $\K G\otimes S(V)$,
which we shortly denote as $\A$.
We now look for pre-Lie products on $\K G\cdot S(V)$, making it a Com-PreLie bialgebra.

\subsection{Several lemmas on $\A$}

\begin{lemma} \label{lem19}
A 1-cocycle of $S(V)$ is a linear map $\phi:S(V)\longrightarrow S(V)$ such that for any $x\in S(V)$,
\begin{align*}
\Delta \circ \phi(x)&=1\otimes \phi(x)+(\phi \otimes \id_{S(V)})\circ \Delta(x).
\end{align*}
\begin{enumerate}
\item If $\dim(V)\geq 2$, for any 1-cocycle $\phi$ of $S(V)$, there exist $\lambda \in \K$ and $F:S^+(V)\longrightarrow \K$ such that
\begin{align}
\nonumber&&\phi(1)&=0,\\
\label{EQ10}&\forall x\in S^+(V),&\phi(x)&=\lambda x+(F\otimes \id_{S(V)})\circ \tdelta(x).
\end{align}
\item If $\dim(V)=1$, let $(X)$ be a basis of $V$. 
There exist scalars $a$, $\lambda$ and a map $F:S^+(V)\longrightarrow \K$ such that
\begin{align*}
&&\phi(1)&=aX,\\
&\forall n\geq 1,&\phi(X^n)&=a\frac{X^{n+1}}{n+1}+\lambda X^n+\sum_{i=1}^{n-1} \binom{n}{i}F(X^i) X^{n-i}.
\end{align*}
\end{enumerate}
\end{lemma}

\begin{proof} \textit{First step.} Let $\lambda \in \K$ and $F:S^+(V)\longrightarrow \K$ be a map.
We consider the map $\phi$ defined on $S(V)$ by  (\ref{EQ10}). Let us prove that it is a 1-cocycle. If $x=1$,
\begin{align*}
\Delta \circ \phi(1)&=0=1\otimes \phi(1)+\phi(1)\otimes 1.
\end{align*}
If $x\in S^+(V)$, then
\begin{align*}
1\otimes \phi(x)+(\phi \otimes \id_{S(V)})\circ \Delta(x)&=1\otimes \phi(x)+\phi(x)\otimes 1+\lambda x'\otimes x''+F(x')x''\otimes x'''\\
&=\Delta(\lambda x+F(x')x'')\\
&=\Delta\circ \phi(x).
\end{align*}

\textit{Second step.} Let $\phi$ be a 1-cocycle and $k\geq 2$ such that $\phi_{\mid S^l(V)}=0$ if $l<k$.
We fix a basis $(e_i)_{i\in I}$ of $V$. Let $x=\prod e_i^{\alpha_i}\in S^k(V)$, with $\sum \alpha_i=k$. Then
\begin{align*}
\Delta \circ \phi(x)&=1\otimes \phi(x)+\phi(x)\otimes 1+\phi(1)\otimes x+\phi(x')\otimes x''=\phi(x)\otimes 1+1\otimes \phi(x),
\end{align*}
so $\phi(x)\in V$. Let us take $j\in J$, we shall consider the linear form defined on $S^{k-1}(V)$ by
\begin{align*}
&\forall x=\prod_{i\in I} e_i^{\beta_i} \in S^{k-1}(V),&F_j(x)&=\frac{1}{\beta_j+1}e_j^* \circ \phi(x e_j).
\end{align*}
Let $\phi_j$ be the 1-cocycle defined in the first section with $\lambda=0$ and $F=F_j$. Then
\begin{align*}
\Delta\circ \phi(xe_j)&=1\otimes \phi(xe_j)+\phi(xe_j)\otimes 1+\sum_{i\in I} \alpha_i \phi\left(\frac{xe_j}{e_i}\right)\otimes e_i+
\phi(x) \otimes e_j.
\end{align*}
As $S(V)$ is cocommutative,
\begin{align*}
(e_j^* \otimes \id_{S(V)})\circ \tdelta\circ \phi(xe_j)&=\sum_{i\in I} \alpha_ie_j^*\circ \phi\left(\frac{xe_j}{e_i}\right)e_i+e_j^*\circ \phi(x) e_j\\
=(\id_{S(V)} \otimes e_j^*)\circ \tdelta\circ \phi(xe_j)&=(\alpha_j+1)\phi(x).
\end{align*}
Hence,
\begin{align*}
\phi(x)&=\sum_{i\neq j} \frac{\alpha_i}{\alpha_j+1} e_j^*\circ \phi\left(\frac{xe_j}{e_i}\right)e_i+e_j^*\circ \phi(x)e_j\\
&=\sum_{i\in I} \alpha_i F_j\left(\frac{xe_i}{e_j}\right) e_i\\
&=\phi_j(x).
\end{align*}
We proved that if $\phi$ is a 1-cocycle such that $\phi_{\mid S^l(V)}=0$ for any $l<k$, with $k\geq 2$, then there exists $F:S^{k-1}(V)\longrightarrow \K$
such that for any $x\in S^k(V)$,
\begin{align*}
\phi(x)&=F(x')\otimes x''.
\end{align*}

\textit{Third step}. Let $\phi$ be a 1-cocycle such that $\phi(1)=0$ and $\phi_{\mid V}=0$.
Let us construct $F_k:S^k(V)\longrightarrow \K$ by induction on $k$, $k\in \N$, such that for any $x\in S^{k+1}(V)$,
\begin{align*}
\phi(x)=((F_1\oplus \ldots \oplus F_k)\otimes \id_{S(V)})\circ \tdelta(x).
\end{align*}
If $k=0$, there is nothing to construct, as $\phi_{\mid V}=0$. Let us assume $F_1,\ldots F_{k-1}$ constructed, with $k\geq 1$.
Let $\psi$ be the 1-cocycle associated in the first step to $\lambda=0$ and $F_1\oplus \ldots \oplus F_{k-1}$.
For any $x\in S^l(V)$, with $l\leq k$, $\phi(x)=\psi(x)$ by the induction hypothesis. By the second step applied to $\phi-\psi$, 
there exists $F_k$ such that for any $x\in S^{k+1}(V)$,
\begin{align*}
\phi(x)-\psi(x)&=F_k(x')x''.
\end{align*}
Hence, for any $x\in S^{k+1}(V)$,
\begin{align*}
\phi(x)=(F_1\oplus\ldots \oplus F_k)(x')x''.
\end{align*}
We proved that for any 1-cocycle $\phi$ such that $\phi(1)=0$ and $\phi_{\mid V}=0$, there exists $F:S^+(V)\longrightarrow \K$
such that for any $x\in S^+(V)$, then $\phi(x)=F(x')x''$.\\

\textit{Fourth step.} Let us consider a 1-cocycle such that $\phi(1)=0$. Let $x\in V$. 
\begin{align*}
\Delta\circ \phi(x)&=1\otimes  \phi(x)+\phi(x)\otimes 1, 
\end{align*}
so $\phi(x)\in V$. If $V$ is 1-dimensional, there exists $\lambda$ such that $\phi_{\mid V}=\lambda \id_V$.
If $\dim(V)\geq 2$, for any $x\in V$,
\begin{align*}
\Delta \circ \phi(x^2)&=1\otimes \phi(x^2)+\phi(x^2)\otimes 1+2\phi(x)\otimes x.
\end{align*}
As $S(V)$ is cocommutative, $\phi(x)$ and $x$ are colinear for any $x\in V$. Hence, there exists $\lambda\in \K$ such that $\phi_{\mid V}=\lambda \id_V$.
Consequently, if $\psi$ is the 1-cocycle associated to $\lambda$ and $F=0$, then $\psi(1)=\phi(1)=0$ and $\phi_{\mid V}=\psi_{\mid V}=\lambda \id_V$.
By the third step applied to $\phi-\psi$, there exists $F$ such that (\ref{EQ10}) holds.\\

\textit{Fifth step}. Let $\varphi$ is a 1-cocycle of $S(V)$, with $\dim(V)\geq 2$. Then
\begin{align*}
\Delta\circ \phi(1)&=\phi(1)\otimes 1+1\otimes \phi(1),
\end{align*}
so $\phi(1)\in V$. Let $x\in V$.
\begin{align*}
\Delta \circ \phi(x)&=1\otimes \phi(x)+\phi(x)\otimes 1+\phi(1)\otimes x.
\end{align*}
As $\Delta$ is cocommutative, $\phi(1)$ and $x$ are colinear, for any $x\in V$. As $\dim(V)\geq 2$, $\phi(1)=0$.
Combined with the fourth step, we obtain point 1.\\

\textit{Last step.} We prove the first assertion of point 2. Let us consider the map $\psi:S(V)\longrightarrow S(V)$, defined by
\begin{align*}
\psi(X^n)&=\frac{X^{n+1}}{n+1}.
\end{align*}
It is a 1-cocycle of $S(V)$, with $V=\vect(X)$.
If $\phi$ is 1-cocycle of $S(V)$, then $\phi(1)$ is a primitive element of $S(V)$, so belong to $\vect(X)$.
Hence, there exists $a\in \K$ such that $\phi(1)=\lambda X$. Then $\phi-a\psi$ is a 1-cocycle of $S(V)$, vanishing on $1$.
We conclude with the fourth step. \end{proof}

\begin{lemma} \label{lem20}
In $\A$:
\begin{enumerate}
\item Let $g,h\in G$, with $g\neq h$.
\begin{align*}
\{x\in \A \mid \Delta(x)=g\otimes x+x\otimes h\}=\vect(g-h).
\end{align*}
\item Let $g\in G$.
\begin{align*}
\{x\in \A\mid \Delta(x)-g\otimes x-x\otimes g\in (\K G\cdot S^+(G))^{\otimes 2}\}=g S^+(V)
\end{align*}\end{enumerate}\end{lemma}

\begin{proof} For both points, the inclusion $\supseteq$ is trivial. Note that
\begin{align*}
\A&=\K G\oplus \K G\cdot S^+(V)=\bigoplus_{k\in G} \K k \oplus \K G\cdot S^+(V).
\end{align*}
For any $k\in G$, we denote by $\varpi_k$ the canonical projection on $\K k$ in this direct sum.\\

1. Let $x\in \A$, such that $\Delta(x)=g\otimes x+x\otimes h$. We put $\varpi_g(x)=\alpha g$. By cocommutativity:
\begin{align*}
(\varpi_g\otimes \id)\circ \Delta(x)&=g\otimes x+\alpha g \otimes h
=(\id \otimes \varpi)\circ \Delta(x)=\alpha g\otimes g+0.
\end{align*}
Hence, $x+\alpha h=\alpha g$, so $x=\alpha(g-h)$.\\

2. We put
\begin{align*}
x&=\sum_{h\in G} hx_h,
\end{align*}
where $x_h\in S(V)$ for any $h\in G$. For any $h\in G$, let us put $\varpi_h(x)=\alpha_h h$. 
\begin{align*}
\Delta(x)&=\sum_{h\in G} h x_h^{(1)}\otimes h x_h^{(2)},\\
(\varpi_h \otimes \id)\circ \Delta(x)&=\delta_{g,h} g \otimes x+\alpha_h h\otimes g=h\otimes x_h.
\end{align*}
For $h\neq g$, we obtain $hx_h=\alpha_h g$, so $x_h=0$ and $\alpha_h=0$. Hence, $x=gx_g\in gS(V)$.
For $h=g$, we obtain $g\otimes gx_g=g\otimes x+\alpha_g g\otimes g$, so $\alpha_g=0$, and $x\in gS^+(V)$. \end{proof}

\subsection{PreLie products on $\A$}

\begin{prop} \label{prop21}
Let $\bullet$ be a product on $\A$, making it a Com-PreLie bialgebra.
\begin{enumerate}
\item For $g\in G\setminus\{1\}$, for any $v\in S(V)$, for any $w\in S(V)$, $v\bullet gw=0$.
\item $S(V)$ is a pre-Lie subalgebra of $\A$.
\item If $\dim(V)\geq 2$ or if the pre-Lie product is nonzero on $S(V)$, then $\K G$ is a pre-Lie subalgebra of $\A$.
\end{enumerate}\end{prop}

\begin{proof}
1. By the Leibniz identity, it is enough to prove it for $v\in V$. For any $k\in \N$, let us prove by induction that
for any $w\in S^l(V)$, with $l<k$, $v\bullet gw=0$. It is trivial if $k=0$. Let us assume the result at rank $k$, $k\in \N$. Let $x\in S^k(V)$.
\begin{align*}
\Delta(v\bullet gw)&=1\otimes v\bullet gw+v\bullet gw^{(1)}\otimes gw^{(2)}\\
&=\begin{cases}
1\otimes v\bullet gw+v\bullet gw\otimes g\mbox{ if }k=0,\\
1\otimes v\bullet gw+v\bullet gw\otimes g+\underbrace{v\bullet g\otimes gw+v\bullet gw'\otimes w''}_{=0\mbox{\scriptsize{ by induction hypothesis}}}
\mbox{ if }k\geq 1.
\end{cases}
\end{align*}
By Lemma \ref{lem20}, there exists a scalar $\lambda(v\otimes w)$ such that $v\bullet gw=\lambda(v\otimes w)(g-1)$.

Let $(e_i)_{i\in I}$ be a basis of $V$ and $w=\prod e_i^{\alpha_i}\in S^k(V)$, with $\sum \alpha_i=k$. For any $i\in I$, 
\begin{align*}
\Delta(v\bullet gwe_i)&=1\otimes v\bullet gwe_i+v\bullet gwe_i\otimes g+\sum_{j\neq i} \alpha_j \lambda\left(v\otimes \frac{we_i}{e_j}\right)
(g-1)\otimes ge_j\\
&+(\alpha_i+1)\lambda(v\otimes w)(g-1)\otimes ge_i.
\end{align*}
Let $\varpi$ be the projector on $\K G\cdot S^+(V)$ which vanishes on $\K G$. By cocommutativity,
\begin{align*}
(\varpi \otimes \id_{\K G\cdot S^+(V)})\circ \Delta(v\bullet gwe_i)&=\varpi(g\otimes we_i) \otimes g\\
=(\id_{\K G\cdot S^+(V)} \otimes \varpi)\circ \Delta(v\bullet gwe_i)&=\sum_{j\neq i} \alpha_j \lambda\left(v\otimes \frac{we_i}{e_j}\right)(g-1)\otimes ge_j\\
&+(\alpha_i+1)\lambda(v\otimes w)(g-1)\otimes ge_i.
\end{align*}
Applying $\varpi_g$, as $\varpi_g \circ \varpi=0$, we obtain that
\begin{align*}
0&=(\varpi_g\otimes \id_{\K G\cdot S^+(V)}) \circ (\varpi \otimes \id_{\K G\cdot S^+(V)})\circ \Delta(v\bullet gwe_i)\\
&=\sum_{j\neq i} \alpha_j \lambda\left(v\otimes \frac{we_i}{e_j}\right)g\otimes ge_j+(\alpha_i+1)\lambda(v\otimes w)g\otimes ge_i.
\end{align*}
Hence, $\lambda(v\otimes w)=0$, so $v\bullet gw=0$.\\

2. By the Leibniz identity, it is enough to prove that $v\bullet w\in S^+(V)$ for any $v\in V$, $w\in S^k(V)$, $k\in \N$, by induction on $k$.
If $k=0$, then
\begin{align*}
\Delta(v\bullet 1)&=v\bullet 1\otimes 1+1\otimes v\bullet 1.
\end{align*}
By Lemma \ref{lem20}, $v\bullet 1\in S^+(V)$. Let us assume the result at all ranks $<k$, with $k\geq 1$.
Then
\begin{align*}
\Delta(v\bullet w)&=v\bullet w\otimes 1+1\otimes v\bullet w+\underbrace{v\bullet w'\otimes w''}
_{\in S^+(V)^{\otimes 2}\mbox{\scriptsize{ by the induction hypothesis}}}.
\end{align*}
By Lemma \ref{lem20}, $v\bullet w\in S^+(V)$.\\

3. Let $g,h\in G$. 
\begin{align*}
\Delta(g\bullet h)&=g\otimes g\bullet h+g\bullet h\otimes gh.
\end{align*}
If $h\neq 1$, by Lemma \ref{lem20}, $g\bullet h\in \prim(g-gh)\subseteq \K G$.
If $h=1$, then $g^{-1}(g\bullet 1)$ is a primitive element of $\A$, so belong to $V$. Let us put $g\bullet 1=gx_g$, with $x_g\in V$.

Let us assume that $\dim(V)\geq 2$. For any $y\in V$,
\begin{align*}
\Delta(g\bullet y)&=g\otimes g\bullet y+g\bullet y\otimes g+g x_g\otimes gy.
\end{align*}
By cocommutativity, $x_g$ and $y$ are colinear, for any $y\in V$. As $\dim(V)\geq 2$, necessarily $x_g=0$, so $g\bullet 1=0\in \K G$. 

Let us assume that $V$ is 1-dimensional, and that the restriction of $\bullet$ to $S(V)$ is nonzero. 
Up to an isomorphism, we replace $S(V)$ by $\K[X]$. By Proposition \ref{prop14},
there are two possibilities.
\begin{enumerate}
\item There exist $\lambda,\mu \in \K$, such that for any $k,l\in \N$,
\begin{align*}
X^k\bullet X^l&=\lambda k l! \sum_{i=1}^{k+l-1}\frac{\mu^{k+l-i-1}}{(i-k+1)!}X^i.
\end{align*}
\item There exists $\lambda \in \K$, 
such that for any $k,l\in \N$,
\begin{align*}
X^k\bullet X^l&=\lambda \frac{k}{l+1}X^{k+l}.
\end{align*}
\end{enumerate}
As $\bullet$ is nonzero, in both cases $\lambda \neq 0$. For any $g\in G$, $g^{-1}(g\bullet 1)\in V=\vect(X)$,
we put $g\bullet 1=\alpha(g) gX$. If $g=1$, as $1\bullet 1=0$, $\alpha(g)=0$. Let us assume that $g\neq 1$. 
Then, by the first point,
\begin{align*}
(X\bullet g)\bullet 1&=0;\\
X\bullet (g\bullet 1)&=\alpha(g) X\bullet X=\begin{cases}
\alpha(g)\lambda X\mbox{ (first case)},\\
\displaystyle \alpha(g) \frac{\lambda}{2}X^2\mbox{ (second case)};
\end{cases}\\
(X\bullet 1)\bullet g&=0,\\
X\bullet (1\bullet g)&=0.
\end{align*}
By the pre-Lie identity, as $\lambda\neq 0$, $\alpha(g)=0$. So $g\bullet 1=0\in \K G$. \end{proof}

\begin{prop}\label{prop22}
Let $\bullet$ be a product on $\A$, making it a Com-PreLie bialgebra.
We assume that $\K G$ is a sub-pre-Lie algebra of $\A$. 
\begin{enumerate}
\item For any $g\in G$, there exist $\lambda(g)\in \K$ and $F_g:S^+(V)\longrightarrow \K$ such that for any $v \in S^+(V)$,
\begin{align*}
g\bullet v=\lambda(g) gv+F_g(v')gv''.
\end{align*}
Moreover, for any $g,h\in G$,
\begin{align*}
\lambda(gh)&=\lambda(g)+\lambda(h),&F_{gh}&=F_g+F_h.
\end{align*}
\item Recall that $S(V)$ is a sub-pre-Lie algebra of $\A$. It is equal to $S(V,f,\lambda)$, with $f\neq 0$, then for any $g\in G$,
there exists $\mu(g)\in \K$ such that for any $x_1,\ldots,x_n \in V$,
\begin{align*}
F_g(x_1\ldots x_n)&=\mu(g) n!\lambda^{n-1} f(x_1)\ldots f(x_n).
\end{align*}
Moreover, for any $g,h\in G$,
\begin{align*}
\mu(gh)&=\mu(g)+\mu(h).
\end{align*}
\item If $(S(V),\bullet)=S(V,f,\lambda)$, with $f\neq 0$, then for any $g,h\in G\setminus\{1\}$, for any $v\in S^+(V)$,
\begin{align*}
g\bullet hv&=-\lambda(g,h)ghv.
\end{align*}
\end{enumerate}\end{prop}

\begin{proof} 1. Let us consider the map
\begin{align*}
\phi_g&:\left\{\begin{array}{rcl}
S(V)&\longrightarrow&\A\\
x&\longrightarrow&g^{-1}(g\bullet x).
\end{array}\right.
\end{align*}
As $\K G$ is a sub-pre-Lie algebra, by Theorem \ref{theo16}, $g\bullet 1=\lambda(g,1)(g-g)=0$, so $\phi_g(1)=0$. 
Moreover, for any $v\in S(V)$,
\begin{align*}
\Delta \circ \phi(v)&=(g^{-1}\otimes g^{-1)})(g\otimes g\bullet v+g\bullet v^{(1)}\otimes g v^{(2)})\\
&=1\otimes \phi_g(v)+\phi_g(v^{(1)})\otimes v^{(2)}.
\end{align*}
By Lemma \ref{lem20}, $\phi_g(v)\in S(V)$, and $\phi_g$ is a 1-cocycle of $S(V)$.
Lemma \ref{lem19} gives the existence of $\lambda(g)$ and $F_g$. 

If $g,h\in G$, for any $v\in S^+(V)$,
\begin{align*}
gh\bullet v&=\lambda(gh)gv+F_{gh}(v')gv''\\
&=(g\bullet v)h+g(h\bullet v)\\
&=(\lambda(g)+\lambda(h))ghv+(F_g(v')+F_h(v'))ghv''.
\end{align*}
Hence,$ \lambda(gh)=\lambda(g)+\lambda(h)$ and $F_{gh}=F_g+F_h$.\\

2. Let $x,y\in S^+(V)$.
\begin{align*}
(g\bullet x)\bullet y&=\lambda(g)^2gxy+\lambda(g)F_g(y')gxy''+\lambda(g)F_g(x')gx''y+F_g(x')F_g(y')gx''y''\\
&+\lambda(g)g(x\bullet y)+F_g(x')g(x''\bullet y),\\
g\bullet (x\bullet y)&=\lambda(g)g(x\bullet y)+F_g(x'\bullet y)gx''+F_g(x\bullet y')gy''+F_g(x'\bullet y')gx''y''+F_g(x')g(x''\bullet y).
\end{align*}
The pre-Lie identity implies that for any $x,y\in S^+(V)$,
\begin{align}
\label{EQ11} &F_g(x'\bullet y)x''+F_g(x\bullet y')y''+F_g(x'\bullet y')x''y''\\
\nonumber&=F_g(y'\bullet x)y''+F_g(y\bullet x')x''+F_g(y'\bullet x')x''y''
\end{align}
Let $(e_i)_{i\in I}$ be a basis of $V$, such that there exists $i_0\in I$, with $f(e_i)=\delta_{i,i_0}$ for any $i\in I$.
By (\ref{EQ11}) for $x=e_{i_0}$, primitive element, for any $y\in S^+(V)$,
\begin{align*}
F_g(e_{i_0}\bullet y')y''&=F_g(y'\bullet e_{i_0})y''.
\end{align*}
Applying $e_{i_0}^*$ on both sides,
\begin{align*}
F_g\left(e_{i_0}\bullet \frac{\partial y}{\partial e_{i_0}}\right)=F_g\left(\frac{\partial y}{\partial e_{i_0}}\bullet e_{i_0}\right).
\end{align*}
By surjectivity of $\dfrac{\partial}{\partial e_{i_0}}$, for any $z\in S^+(V)$,
\begin{align*}
F_g(z\bullet e_{i_0})&=F_g(z\bullet e_{i_0}).
\end{align*}
Let us consider $\displaystyle z=\prod e_i^{\alpha_i}$, with $\sum \alpha_i=n\geq 1$. Then
\begin{align*}
z\bullet e_{i_0}&=\alpha_{i_0} z,\\
e_{i_0}\bullet z&=\begin{cases}
\displaystyle \sum_{k=0}^{\alpha_{i_0}} \frac{n!}{(n-k)!} \lambda^ke_{i_0}^{\alpha_{i_0}-k}\prod_{i\neq i_0} e_i^{\alpha_i} \mbox{ if }n\neq \alpha_{i_0},\\
\displaystyle \sum_{k=0}^{\alpha_{i_0}-1} \frac{n!}{(n-k)!} \lambda^ke_{i_0}^{\alpha_{i_0}-k}\prod_{i\neq i_0} e_i^{\alpha_i} \mbox{ if }n=\alpha_{i_0}.
\end{cases} \end{align*}
Noticing that $F_g(1)=0$, we obtain that
\begin{align}
\label{EQ12}F_g\left((1-\alpha_{i_0})z+\sum_{k=1}^{\alpha_{i_0}}\frac{n!}{(n-k)!} \lambda^ke_{i_0}^{\alpha_{i_0}-k}\prod_{i\neq i_0} e_i^{\alpha_i}\right)&=0.
\end{align}

We put $\mu(g)=F_g(e_{i_0})$. Let us prove that for any $n \geq 1$,
\begin{align*}
F_g(e_{i_0}^n)&=n! \lambda^{n-1}\mu(g)=n! \lambda^{n-1}\mu(g)f(e_{i_0})^n.
\end{align*}
This is obvious if $n=1$. If the result is true at all ranks $1\leq k<n$, we obtain from (\ref{EQ12}) that
\begin{align*}
0&=(1-n)F_g(e_{i_0}^n)+\sum_{k=1}^{n-1}\frac{n!}{(n-k)!} \lambda^k (n-k)!\lambda^{n-k-1}\mu(g)\\
&=(1-n)F_g(e_{i_0}^n)+(n-1)n! \lambda^{n-1}\mu(g).
\end{align*}
This implies the result at rank $n$. \\

Let us consider $j\in I\setminus \{i_0\}$. By (\ref{EQ11}), for $x=e_j$, primitive, for any $y\in S^+(V)$,
\begin{align*}
F_g(y'\bullet e_j)y''&=0.
\end{align*}
Applying $e_j^*$, for any $y\in S^+(V)$,
\begin{align*}
F_g\left(\frac{\partial y}{\partial e_j}\bullet e_j\right)&=0.
\end{align*}
By surjectivity of $\frac{\partial}{\partial e_j}$, for any $z\in S^+(V)$,
\begin{align*}
F_g(z\bullet e_j)&=0.
\end{align*}
If $z=\prod e_i^{\alpha_i}$, then
\begin{align*}
z\bullet e_j&=\alpha_{i_0} e_j e_{i_0}^{\alpha_{i_0}-1}\prod_{i\in I\setminus\{i_0\}}e_i^{\alpha_i}.
\end{align*}
Consequently, if $y=\prod e_i^{\alpha_i} \in S^+(V)$, with $j\neq i_0$ such that $\alpha_j\neq 0$, denoting $n=\sum \alpha_i$,
\begin{align*}
F_g(y)&=0=\mu(g)n!\lambda^{n-1} \prod_{i\in I} f(e_i)^{\alpha_i},
\end{align*}
as $f(e_j)=0$ and $\alpha_j\geq 1$. Finally, for any $y=\prod e_i^{\alpha_i} \in S^+(V)$, 
denoting $n=\sum \alpha_i$,
\begin{align*}
F_g(y)&=\mu(g)n!\lambda^{n-1} \prod_{i\in I} f(e_i)^{\alpha_i}.
\end{align*}
As $F_{gh}=F_g+F_h$, $\mu(gh)=F_{gh}(e_{i_0})=F_g(e_{i_0})+F_h(e_{i_0})=\mu(g)+\mu(h)$.\\

3. For any $v\in S^+(V)$, we put $\varpi_g(g\bullet hv)=\alpha(v)g$, where $\varpi_g$ is defined in the proof of Lemma \ref{lem20}.
Then
\begin{align*}
\Delta(g\bullet hv)&=g\bullet hv\otimes gh+g\bullet h\otimes ghv+g\bullet hv'\otimes ghv''+g\otimes g\bullet hv\\
&=g\bullet hv\otimes gh+\lambda(g,h)(g-gh)\otimes ghv+g\bullet hv'\otimes ghv''+g\otimes g\bullet hv.
\end{align*}
By cocommutativity,
\begin{align*}
(\varpi_g \otimes \id)\circ\Delta(g\bullet hv)&=\alpha(v)g\otimes gh+\lambda(g,h)g\otimes ghv+\alpha(v')g\otimes ghv''+g\otimes g\bullet hv\\
=(\id \otimes \varpi_g)\circ\Delta(g\bullet hv)&=\alpha(v) g\otimes g.
\end{align*}
This gives
\begin{align*}
g\bullet hv&=-\lambda(g,h)ghv+\alpha(v)(g-gh)-\alpha(v')ghv''.
\end{align*}
It remains to prove that the linear form $\alpha$ is zero. Let $x,y\in S^+(V)$.
\begin{align*}
(g\bullet hx)\bullet y&\in \K G\cdots S^+(V),\\
g\bullet (hx\bullet y)&=\left(\lambda(h)\alpha(xy)+\lambda(h)\alpha(xy'')F_g(y')+\alpha(x\bullet y)\right)(g-gh)+\mbox{terms in }\K G\cdots S^+(V),\\
(g\bullet y)\bullet hx&\in  \K G\cdots S^+(V),\\
g\bullet (y\bullet hx)&\in  \K G\cdots S^+(V).
\end{align*}
By the pre-Lie identity, for any $x,y\in S^+(V)$,
\begin{align}
\label{EQ13} \alpha\left(\lambda(h)xy+F_h(y')xy''+x\bullet y\right)&=0.
\end{align}

\textit{First sub-case}. We assume that $\lambda(h)=0$. For $x\in S^+(V)$ and $y\in V$, we obtain that
\begin{align*}
\alpha(S^+(V)\bullet V)&=(0).
\end{align*}
Let $y=\prod e_i^{\alpha_i}\in S^+(V)$, with $\sum \alpha_i=n\geq 1$. If $\alpha_{i_0}\geq 1$, then
\begin{align*}
y\bullet e_{i_0}&\alpha_{i_0}y,
\end{align*}
so $y\in S^+(V)\bullet V$. Otherwise, there exists $j\neq i_0$, such that $\alpha_j\geq 1$. 
\begin{align*}
\left(e_{i_0} e_j^{\alpha_j-1}\prod_{i\neq i_0,j} e_i^{\alpha_i}\right)\bullet e_j&=(\alpha_{i_0}+1) y=y,
\end{align*}
so $y\in S^+(V)\bullet V$. As a conclusion, $S^+(V)\bullet V=S^+(V)$, so $\alpha=0$.\\

\textit{Second sub-case}. We assume that $\lambda(h)\neq 0$. Let us first prove that the ideal $I$ generated by $\ker(f)$ is a subspace of $\ker(\alpha)$.
Let $x_1,\ldots,x_n \in V$ and $y\in \ker(f)$, let us prove that $x_1\ldots x_n y\in \ker(\alpha)$ by induction on $n$.
If $n=0$, by (\ref{EQ13}) with $x\in V$,
\begin{align*}
\alpha(\lambda(h)xy+x\bullet y)&=\alpha(\lambda(h)xy+y\bullet x)=0.
\end{align*}
Hence, $x\bullet y-y\bullet x=f(x)y-f(y)x \in \ker(\alpha)$. Choosing $g$ such that $f(x)=1$ we obtain $y\in \ker(\alpha)$.
Let us assume the result at rank $n-1$, $n\geq 1$. The following element belongs to $\ker(\alpha)$ by (\ref{EQ13}), with $x=x_1\ldots x_n$,
\begin{align*}
\lambda(h) x_1\ldots x_n y+\sum_{i=1}^n f(x_i) x_1\ldots x_{i-1}yx_{i+1}\ldots x_n.
\end{align*}
Applying the induction hypothesis and $\lambda(h)\neq 0$, we obtain that $x_1\ldots x_n y\in \ker(\alpha)$.\\

Consequently, there exists a family of scalars $(\beta(n))_{n\geq 1}$ such that for any $x_1,\ldots,x_n \in V$,
\begin{align*}
\alpha(x_1\ldots x_n)&=\beta(n) f(x_1)\ldots f(x_n),
\end{align*}
with the notations of the proof of point 2, $\beta(n)=\alpha(e_{i_0}^n)$. By (\ref{EQ13}), with $x=e_{i_0}^n$ and $y=e_{i_0}$,
for any $n\geq 1$,
\begin{align*}
\lambda(h) \beta(n+1)+n\beta(n)&=0,
\end{align*}
which gives, for any $n\geq 1$,
\begin{align*}
\beta(n)&=\frac{(-1)^{n-1} (n-1)!}{\lambda(h)^{n-1}}\beta(1).
\end{align*}
Let us assume that $\beta(1)\neq 0$. By (\ref{EQ13}) with $x=e_{i_0}$ and $y=e_{i_0}^2$, we obtain that
\begin{align*}
\frac{\beta(1)}{\lambda(h)}(-2\mu(g)+2\lambda \lambda(h)+1)&=0.
\end{align*}
Hence, $\mu(g)=\lambda \lambda(h)+\frac{1}{2}$. By (\ref{EQ13}) with $x=e_{i_0}$ and $y=e_{i_0}^3$, we obtain that
\begin{align*}
\frac{\beta(1)}{\lambda(h)^2}&=0.
\end{align*}
This is a contradiction, so $\beta(1)=0$ and, therefore, $\alpha=0$. \end{proof}

\begin{theo}\label{theo23}
Let $\bullet$ be a product on $\A$, making it a Com-PreLie bialgebra. We assume that the restriction of $\bullet$ to $S(V)$ is nonzero and that
one of the following assertions holds:
\begin{enumerate}
\item $\dim(V)\geq 2$.
\item $\dim(V)=1$ and $v\bullet 1=0$ for any $v\in V$.
\end{enumerate}
Then
\begin{enumerate}
\item There exist $f:V\longrightarrow \K$, nonzero, and $\lambda \in \K$ such that the Com-PreLie Hopf subalgebra $(S(V),m,\bullet,\Delta)$ 
is equal to $S(V,f,\lambda)$.
\item There exist a family of scalars $(\lambda(g,h))_{g,h\in G}$ satisfying the following conditions:
\begin{align*}
&\forall g\in G,&\lambda(g,1)&=0.\\
&\forall g,h,k\in G,&\lambda(gh,k)&=\lambda(g,k)+\lambda(h,k).\\
&\forall g,h,k\in G\mbox{ with }kh\neq 1,&&\lambda(g,hk)\lambda(h,k)-\lambda(g,h)\lambda(h,k)\\
&&&=\lambda(g,hk)\lambda(k,h)-\lambda(g,k)\lambda(k,h).\\
&\forall g,h\in G,&\lambda(g,g)&=0.
\end{align*}
such that for any $g,h\in G$, $x\in S^+(V)$,
\begin{align*}
g\bullet h&=\lambda(g,h)(g-gh),&g\bullet hx&=-\lambda(g,h)ghx.
\end{align*}
\item For any $x\in S(V)$, $g\in G\setminus \{1\}$, $y\in S(V)$, $x\bullet gy=0$. 
\item There exist group morphisms $\lambda,\mu:G\longrightarrow (\K,+)$ such that for any $g\in G$, $x_1,\ldots,x_n \in V$, $n\geq 1$,
\begin{align*}
g\bullet x_1\ldots x_n&=\lambda(g) gx_1\ldots x_n+\mu(g)g\sum_{\emptyset \subsetneq I\subsetneq [n]}
|I|! \lambda^{|I|-1} \prod_{i\in I} f(x_i) \prod_{i\notin I} x_i.
\end{align*}
Conversely, if one define a product $\bullet$ on $\A$ with the point 1.-4. and, for any $g\in G$, $x\in S(V)$, $y\in \A$,
\begin{align*}
gx\bullet y&=(g\bullet y)x+g(x\bullet y),
\end{align*}
then $(\A,m,\bullet,\Delta)$ is a Com-PreLie bialgebra.
\end{enumerate}\end{theo}

\begin{proof}
Let $\bullet$ be a product on $\A$, making it a Com-PreLie bialgebra.By Proposition \ref{prop21}-2, 
we obtain that $S(V)$ is a Com-PreLie subalgebra of $\A$. Theorem \ref{theo9} and the hypotheses on $S(V)$ imply that 
$(S(V),\bullet)=S(V,f,\lambda)$ for well-chosen $f$ and $\lambda$, which is point 1.
By Proposition \ref{prop21}-3, $\K G$ is a Com-PreLie subalgebra of $\A$. Theorem \ref{theo16} gives the existence of the scalars
$\lambda(g,h)$ satisfying the first four points, such that for any $g,h\in G$, $g\bullet h=\lambda(g,h)(g-gh)$.
By Proposition \ref{prop22}-3, for any $g,h\in G\setminus\{1\}$, $v\in S^+(V)$, $g\bullet hv=-\lambda(g,h)ghv$,
which gives point 2. Point 3 comes from Proposition \ref{prop21}, together with the Leibniz identity,
and Point 4 from Proposition \ref{prop22}-2. 

Conversely, given $\lambda \in \K$, $f:V\longrightarrow \K$, $\mu,\lambda:G\longrightarrow (\K,+)$ be group morphisms,
and scalars $(\lambda(g,h))_{g,h\in G}$ satisfying the two first conditions of point 2, points 1-4 define a product $\bullet$ on $\A$ 
satisfying the Leibniz identity and the compatibility with the coproduct. 
It remains to prove the pre-Lie identity. Because of the Leibniz identity, it is enough to prove it in the following cases:
\begin{enumerate}
\item $x\in V$, $y\in S(V),z\in S(V)$. This comes from Theorem \ref{theo2}. 
\item $x\in G\setminus \{1\}$, $y,z\in G$. By Theorem \ref{theo16}, this holds if, and only if, the third first conditions of point 2 are satisfied.
We now assume that these conditions hold. 
\item $x\in V$, $y\in G$. This is immediate, as all terms in the pre-Lie identity are zero.
\item $x\in V$, $y\in S(V)$, $z\in G\setminus\{1\}\cdots S^+(V)$. This is immediate, as all terms in the pre-Lie identity are zero.
\item $x\in V$, $y,z\in G\setminus\{1\}\cdots S^+(V)$. We put $y=gv$ and $z=hw$. If $gh\neq 1$, all terms in the pre-Lie identity are zero.
Otherwise,
\begin{align*}
x\bullet (gv\bullet hw)-(x\bullet gv)\bullet hw&=\lambda(g,h)x\bullet vw,\\
x\bullet (hw\bullet gv)-(x\bullet hw)\bullet gvw&=\lambda(h,g)x\bullet vw.
\end{align*}
As $\lambda(h^{-1},h)=-\lambda(h,h)$.
For well-chosen $x,v,w$, $x\bullet vw\neq 0$, so the pre-Lie identity holds in this case if, and only if, for any $h\in G$,
\begin{align}
\label{EQ14}\lambda(h,h^{-1})&=-\lambda(h,h).
\end{align}
We now assume that this condition holds. Note that it is implies by the fourth condition of point 2.
\item $x\in G\setminus\{1\}$, $y\in S(V)$: this is proved by direct computations, separating the cases $z\in S(V)$, $z\in G$ and
$z\in G\setminus\{1\}\cdots S^+(V)$.
\item $x,y\in G\setminus \{1\}$, $z\in G\setminus\{1\}\cdots S^+(V)$. We put $x=g$, $y=hv$, $z=kw$.
If $hk\neq 1$, this is proved by direct computations. If $hk=1$, the pre-Lie identity is satisfied if, and only if
\begin{align*}
\lambda(h,k)(\lambda(g,h)-\lambda(g,k))&=0.
\end{align*}
In particular, if $g=h$, by (\ref{EQ14}),
\begin{align*}
\lambda(h,k)(\lambda(g,h)-\lambda(g,k))&=-\lambda(h,h)(\lambda(h,h)+\lambda(h,h))=2\lambda(h,h)^2=0,
\end{align*}
so $\lambda(h,h)=0$: this is the fourth condition of point 2.
\item $x\in G\setminus \{1\}$, $y,z\in G\setminus\{1\}\cdots S^+(V)$. If the conditions of point 2 hold, a direct computation shows that
the pre-Lie identity is satisfied in this case.
\end{enumerate}
Finally, if $\bullet$ gives $\A$ a Com-PreLie bialgebra structure, then points 1-4 are satisfied, and conversely. \end{proof}

Let us now consider the coefficients $\lambda(g,h)$ satisfying the conditions of point 2 if $G=\Z$.

\begin{prop}
Let $(\lambda(g,h))_{g,h\in \Z}$ be coefficients satisfying the conditions of point 2 of Theorem \ref{theo23}, with $G=\Z$.
Then for any $g,h\in \Z$, $\lambda(g,h)=0$.
\end{prop}

\begin{proof} We use Theorem \ref{theo18}. In the second case, there exist $k_0\in \Z$, nonzero, $a\in \K$, nonzero, such that for any $k,l\in \Z$,
\begin{align*}
\lambda(k,l)&=a\delta_{l,k_0} k.
\end{align*}
Then $\lambda(k_0,k_0)=ak_0=0$, so $k_0=0$: this is a contradiction. 
In the third case, there exist nonzero scalars $\alpha,\beta$ and $N\geq 1$ such that for any $k,l\in \Z$,
\begin{align*}
\lambda(k,l)&=\begin{cases}
\displaystyle\frac{\alpha\beta}{\frac{l}{N}(\alpha-\beta)+2\beta-\alpha}\mbox{ if }N\mid l,\\
0\mbox{ otherwise}.
\end{cases}
\end{align*}
Then
\begin{align*}
\lambda(N,N)&=\alpha=0.
\end{align*}
This is a contradiction. Consequently, $\lambda(k,l)=0$ for any $k,l\in \Z$.
\end{proof}

\bibliographystyle{amsplain}
\bibliography{biblio}

\end{document}